\documentclass[preprint,10pt]{elsarticle}
\usepackage[utf8]{inputenc}
\usepackage[left=2cm,right=2cm,top=3cm,bottom=2cm]{geometry}
\usepackage[T1]{fontenc} 
\usepackage{lmodern}

\usepackage{lineno}
\modulolinenumbers[5]

\usepackage{amsmath}
\usepackage{amsthm}
\usepackage{amsfonts}
\usepackage{amssymb}
\usepackage{array}
\usepackage{caption}
\usepackage{enumitem}
\usepackage{framed}
\usepackage{graphicx}
\usepackage{pdfpages}
\usepackage{subcaption}
\usepackage{wrapfig}
\usepackage{xcolor}

\newtheorem{thm}{Theorem}
\newtheorem{lem}[thm]{Lemma}
\newtheorem{cor}[thm]{Corollary}
\newdefinition{rmk}[thm]{Remark}
\newdefinition{defn}[thm]{Definition}
\newdefinition{exmp}[thm]{Example}

\newcommand{\Real}{\mathbb{R}}
\newcommand{\rng}{\mathrm{Rng}}

\usepackage{multirow}

\usepackage{natbib}
\usepackage{etoolbox}
\apptocmd\normalsize{%
 \abovedisplayskip=4pt
 \abovedisplayshortskip=3pt
 \belowdisplayskip=4pt
 \belowdisplayshortskip=6pt
}{}{}

\usepackage{enumitem}

\setlist{topsep=5pt,itemsep=2pt} 

\usepackage[color=blue2]{todonotes}
\newcommand{\remPS}[1]{\noindent {\color{cyan}[ps: #1]}}

\newcommand{\remJV}[1]{\noindent {\color{green}[jv: #1]}}
\newcommand{\remJH}[1]{\noindent {\color{olive}[jh: #1]}}

\usepackage{mathrsfs}

\newcommand{\mir}{\varphi}

\usepackage{accents}

\DeclareMathOperator{\image}{Im}

\newcommand{\alphabet}{\mathcal{A}_3}

\newcommand{\N}{\mathbb{N}}
\newcommand{\Z}{\mathbb{Z}}
\newcommand{\fstrip}{V^\times}
\newcommand{\fstripL}{\fstrip_1}
\newcommand{\fstripR}{\fstrip_2}
\newcommand{\sector}{\mathcal{S}}

\newcommand{\ru}{\mathrm{u}}
\newcommand{\rv}{\mathrm{v}}

\begin{document}

\begin{frontmatter}
\title{
Characterization of spatial topological chaos for bistable lattice equation
}

\author[1]{Jakub Hesoun}
\ead{hesounj@fav.zcu.cz}

\author[1]{Petr Stehl\'{i}k\corref{cor1}}
\ead{pstehlik@fav.zcu.cz}

\author[1]{Jon\'a\v{s} Volek}
\ead{volek1@fav.zcu.cz}

\cortext[cor1]{Corresponding author}

\address[1]{Department of Mathematics and NTIS, University of West Bohemia, Univerzitn\'{i} 8, 301~00 Pilsen, Czech Republic}

\begin{abstract}

In this paper we extend the description of spatial topological chaos for bistable lattice equations in a small diffusion regime. We generalize the Keener's proof and show that stationary solutions are topologically conjugate to biinfinite sequences on a three-symbol alphabet. We provide a priori estimates for the values of stationary solutions and use them to show that there are exactly two monotone stationary fronts (up to a translation). Finally, we apply the localization of all stationary solutions to characterize their stability and get the existence of unique stable and unstable monotone fronts (up to a translation).

\end{abstract}


\end{frontmatter}

\section{Introduction and main theorems}
\label{sec:intro}

We characterize and localize $3^\Z$ stationary solutions of Nagumo-type bistable lattice differential equations (LDEs)
\begin{equation}\label{e:lde:Nagumo} 
{u}_{i}'(t) = d(u_{i-1}(t) - 2 u_{i}(t) + u_{i+1}(t))+g(u_{i}(t);a),\quad i\in\mathbb{Z}, \quad t \geq 0,
\end{equation}
in which $ d>0 $ is a sufficiently small diffusion coefficient and $g$ is a smooth bistable reaction function satisfying:
\begin{enumerate}[label=\itshape{(g\textsubscript{\arabic*})}]
    \item \label{hyp:g:nodes} $ g(0) = g(a) = g(1) = 0 $ in which $ a \in (0,1) $,
    \item \label{hyp:g:shape} $ g'(u) < 0 $ for $ u \in [0,a_{1}) \cup (a_{2},1] $ and $ g'(u) > 0 $ for $ u \in (a_{1},a_{2}) $ in which $ 0 < a_{1} < a < a_{2} < 1 $.
\end{enumerate}

As a special case, we study monotone stationary fronts $u=(u_i)$, $i\in\Z$, of \eqref{e:lde:Nagumo}, i.e., solutions of the second-order difference equation
\begin{equation}\label{e:stationary}
d(u_{i-1}-2u_{i}+u_{i+1}) + g(u_{i}) = 0, \quad i \in \mathbb{Z},
\end{equation}
that monotonically connect two stable homogeneous solutions $u\equiv 0$ and $u\equiv 1$, i.e.,
\begin{equation}\label{e:connection}
\lim_{i\to{-\infty}}u_i=0,\quad \lim_{i\to\infty} u_i=1, \text{ and } \quad u_{i}\leq u_{i+1} \text{ for all } i\in\mathbb{Z}.
\end{equation}

The translational invariance of the problem~\eqref{e:stationary} implies that if $(u_i)$, $i\in\Z$, is a stationary solution of~\eqref{e:lde:Nagumo} then $(v_i)$, $i\in\Z$, with $v_i=u_{i+k}$, $k\in\mathbb{Z}$, is also a stationary solution. In this paper, we show that there are exactly two (up to a translation) different monotone stationary profiles of \eqref{e:lde:Nagumo} satisfying \eqref{e:connection}.

\paragraph{Planar maps and spatial topological chaos} Stationary solutions (not only monotone fronts) of the LDE~\eqref{e:lde:Nagumo} have been extensively studied by  planar maps derived from~\eqref{e:stationary}. Keener \cite{Keener1987} considered the map $\phi:\mathbb{R}^2\to\mathbb{R}^2$, $(u_{n+1},v_{n+1})=\phi(u_n,v_n)$ in the form 
\begin{equation}\label{e:map:Keener}
\begin{cases}
    u_{n+1}=2u_n-v_n-\tfrac{1}{d}g(u_n),\\
    v_{n+1}=u_n.
\end{cases}
\end{equation}
He applied a modification of Moser's theorem \cite{Moser1975} (Moser's theorem being a variant of Smale's horseshoe theorem \cite{Smale1967}) to get that the map $\phi$ \eqref{e:map:Keener} possesses the shift $\sigma$ on the biinfinite sequence of symbols $\{0,1 \}$ as a subsystem. Nagumo-type LDEs with sufficiently small $d$ are thus the key example of the spatial topological chaos, \cite{Chow1995}. Keener's results show that their stationary solutions form an invariant set that is topologically semiconjugate to the chaotic symbolic shift $\sigma$. 




\paragraph{Pinning} Keener's analysis \cite{Keener1987} of the map~\eqref{e:map:Keener} and the large number of stationary solutions in a small diffusion regime is closely tied to the description of the pinning phenomenon. The Nagumo partial differential equation
\[
u_t=du_{xx}+g_{\mathrm{cub}}(u;a),
\]
with the cubic bistability $g_{\mathrm{cub}}(u;a)=u(1-u)(u-a)$, $a\in(0,1)$, has a monotone stationary front if and only if $a=1/2$. Moreover, the monotone front has the explicit description and speed $c$
$$
u(x,t)=\frac{1}{2}\left(1+\tanh\left(\frac{x-ct}{2\sqrt{2d}}\right)\right), \quad c=\sqrt{2d}\left(a-\frac{1}{2}\right).
$$
In contrast, the Nagumo LDE~\eqref{e:lde:Nagumo} with the cubic $g=g_{\mathrm{cub}}$ admits propagation failure -- for each $a\in(0,1)$ the monotone fronts do not move for sufficiently small diffusion $d>0$ and start moving only if the diffusion $d$ is large enough $d>D_c(a)$, $a\neq 1/2$. 

Keener \cite{Keener1987} also gave the first estimate of the threshold $D_c(a)$ for small values $a\approx0$. Various approaches have been used to improve these estimates on $D_c(a)$ and extend them for all values of $a\neq1/2$, e.g., \cite{Bustamante2025, Comte2001, Erneux1993}.

Propagation failure is a generic phenomenon for bistable LDEs \eqref{e:lde:Nagumo} and has been observed, e.g., in coupled chemical reactors \cite{Laplante1992}, trapping of light waves in optical lattices \cite{Yablonovitch1999}, or inconsistencies in propagation of nerve signals \cite{Lucchinetti2008}. 

However, the exact bounds $D_c(a)$ are known only in special cases of the so-called piecewise-linear bistable caricatures. Fath \cite{Fath1998} computed the bounds $D_c(a)$ for the LDE~\eqref{e:lde:Nagumo} with the discontinuous McKean's caricature
$$
g_\mathrm{MK}(u;a) = \begin{cases}
-u, & u\leq a,\\
1-u, & u>a.
\end{cases}
$$
Similarly, Elmer \cite{Elmer2005, Elmer2006} described curves $D_c(a)$ for the LDE~\eqref{e:lde:Nagumo} with the continuous but nonsmooth sawtooth caricature
\begin{equation}\label{e:sawtooth}
g_\mathrm{ST}(u;a) = \begin{cases}
 -  u, & u\leq a/2,\\
 (u-a), & u\in(a/2,(a+1)/2),\\
 -(u-1), & u\geq (a+1)/2.
\end{cases}
\end{equation}

Bistable LDEs~\eqref{e:lde:Nagumo} with an additional forcing term have been recently studied to describe the external forces necessary to remove the propagation failure phenomenon in particle chains, \cite{AlHaj2023, Zhou2025}.

\paragraph{Monotone and Periodic Patterns} Among the large number of stationary solutions, monotone and periodic ones have received special treatment. Elmer \cite{Elmer2005} showed that for $d>0$ sufficiently small there is an explicit interval $a=[a_-,a_+]$ for which the LDE~\eqref{e:lde:Nagumo} with the sawtooth bistable caricature $g=g_\mathrm{ST}$ has exactly two monotone stationary fronts (up to a translation). Moreover, there naturally exist nonmonotone pinned and traveling waves in bistable LDEs~\eqref{e:lde:Nagumo} that connect homogeneous stationary solutions with periodic stationary solutions and can create intricate wave collisions \cite{Hupkes2019, Hupkes2019b}.

The periodic patterns of the LDE~\eqref{e:lde:Nagumo} themselves possess interesting properties for small $d$. An $n$-periodic stationary solution can be represented by a word $w=\{0,a,1\}^n$ of length $n\in\N$. Consequently, there are $3^n$ of $n$-periodic solutions, $2^n$ of them corresponding to words $w=\{0,1\}^n$ are asymptotically stable, \cite{Hupkes2019c}. Moreover, there is a partial ordering among the periodic patterns. Every pair of $n$-periodic patterns can be ordered if at least one of them is stable, i.e., represented by a word $w=\{0,1\}^n$. This is no longer true if both solutions are unstable, i.e., represented by words from $\{0,a,1\}^n$, these solutions do not necessarily preserve ordering of the corresponding symbolic words, \cite{Hupkes2019, Hupkes2019c}.




\begin{figure*}[t!]
\centering
        \includegraphics[width=0.9\linewidth]{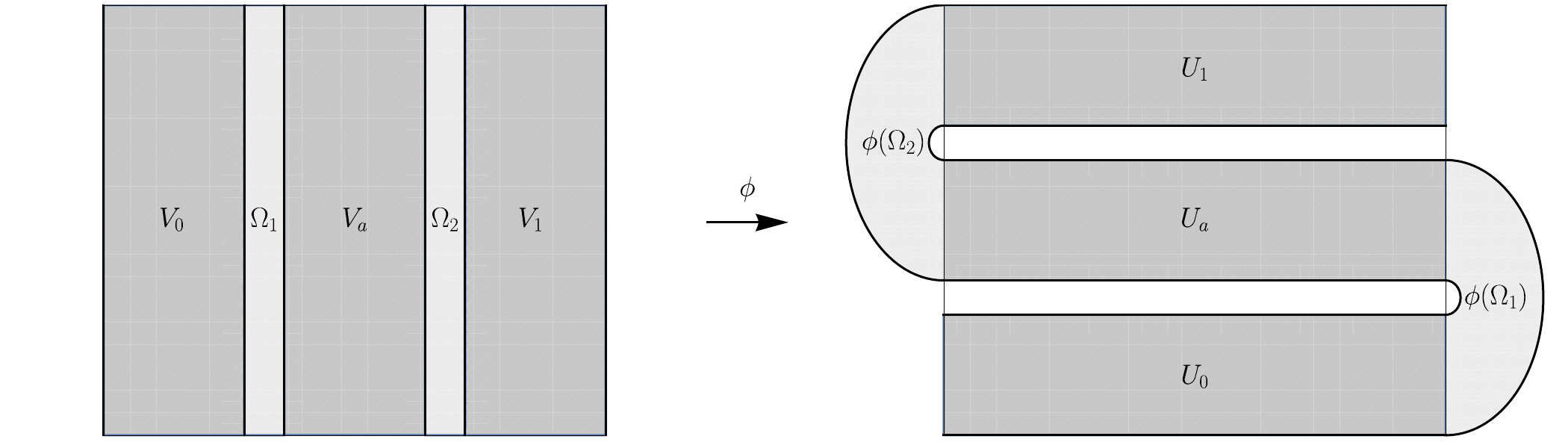}\\
        \includegraphics[width=0.9\linewidth]{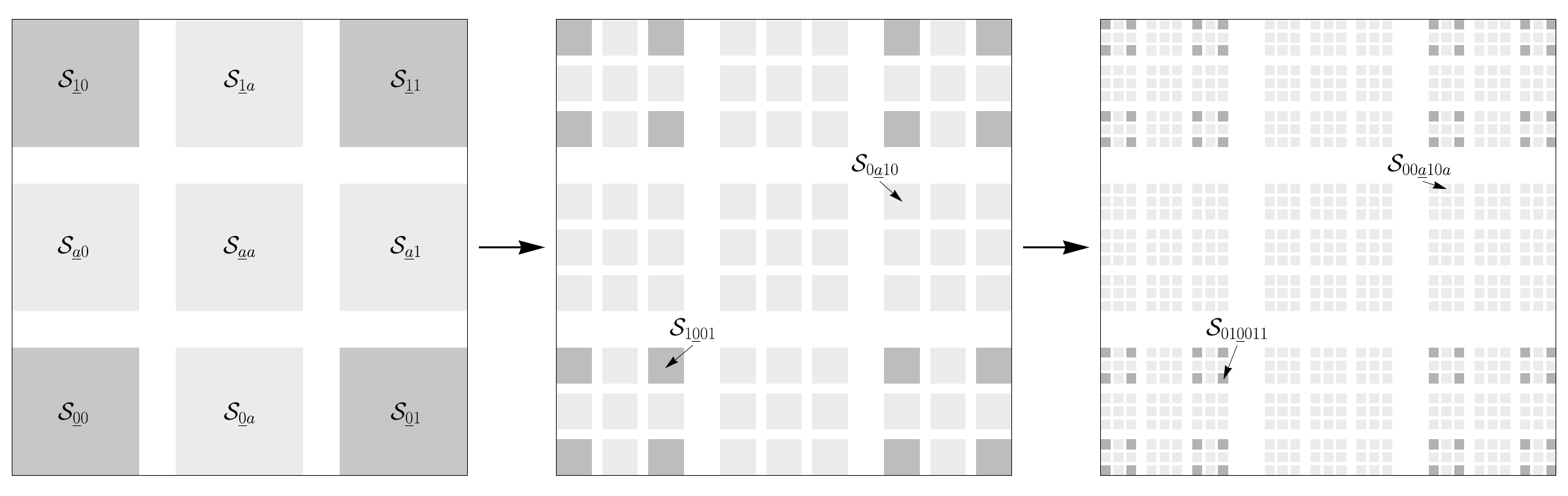}
        \caption{Simplified schematic illustration of the symbolic dynamics connected to the map $\phi$~\eqref{e:map:Keener} and Theorems~\ref{t:main-uniqueness}--\ref{t:main-stability}. The unit square $Q=[0,1]^2$ is stretched horizontally and bent twice so that the vertical strips $V_0$, $V_a$, $V_1$ are mapped onto horizontal strips $U_0$, $U_a$, $U_1$, respectively (top panel). The iterated forward and backward applications of the map $\phi$ creates $3^{2n}$, $n\in\N$, sectors corresponding to words $w\in\{0,a,1\}^{2n}$ (bottom panel). Out of these, $2^{2n}$ sectors are generated by the two-letter alphabet (dark gray sectors). The naming of these sectors is included here for illustration and will be explained in detail in Section~\ref{sec:monotonicity}. Real nonlinear versions of these images generated by the map $\phi$~\eqref{e:map:Keener} contain thinner strips and tiny sectors, see~Figure~\ref{fig:1sectors} for a comparison.}
        \label{fig:symbolic-caricature}
\end{figure*}

\paragraph{Main results} In this paper we extend Keener's argument \cite{Keener1987} but apply the full version of Moser's horseshoe theorem \cite{Moser1975} to show that the map \eqref{e:map:Keener} is topologically conjugate to the shift on the three-letter biinfinite sequences $\Sigma_3=\{0,a,1\}^\Z$, see Figure~\ref{fig:symbolic-caricature}. Therefore, there are $3^\Z$ stationary solutions $u=(u_i)$, $i\in\Z$, of \eqref{e:lde:Nagumo} and these solutions generate points $(u_{i+1},u_i)$, $i\in\Z$, that form a Cantor set  $\Lambda\subset Q=[0,1]^2$.
\begin{thm}[Topological conjugacy to biinfinite sequences on three-symbol alphabet] \label{t:main-uniqueness}
Let $g \in C^{1}([0,1]) $ be a bi\-stable nonlinearity satisfying \ref{hyp:g:nodes}--\ref{hyp:g:shape}. Then for every sufficiently small $ d > 0 $, the planar map $\phi:\Real^2\to\Real^2$ defined by~\eqref{e:map:Keener} possesses the shift $\sigma:\Sigma_3\to\Sigma_3$ as a subsystem, i.e., there exists a set $\Lambda\subset Q=[0,1]^2$ and a homeomorphism $\tau:\Sigma_3\to \Lambda$ such that $\phi(\tau(s)) = \tau(\sigma(s))$ for every $s\in\Sigma_3$.
\end{thm}
In contrast to Keener's paper \cite{Keener1987}, which described the topological semiconjugacy with the two-letter alphabet $\Sigma_2=\{0,1\}^\Z$, this statement with the three-letter alphabet $\Sigma_3$ extends the result so that $\tau$ is a homeomorphism. 



Further, we provide a priori estimates for the position of $\tau(s)$ in the square $Q=[0,1]^2$ for a given biinfinite symbolic sequence $s\in\Sigma_3$. This enables us to get a priori estimates for the values $u_i$ of stationary solutions $u=(u_i)$. The localization can be made arbitrarily precise based on the subsequences of symbols from $\Sigma_3$. On the roughest scale, single letters  $s_i = 0,a,1$ imply $u_i\in[0,a_1)$, $u_i\in(a_1,a_2)$, $u_i\in(a_2,1)$, respectively. More detailed arguments involving words $w\in\{0,a,1\}^n$ of length $n\in\N$ provide more precise estimates. Among other things, we show that there are exactly two monotone fronts (up to a translation) satisfying~\eqref{e:connection}.

\begin{thm}[Exactly two monotone pinned waves] \label{t:main-twosol}
Let $g \in C^{1}([0,1]) $ be a bistable nonlinearity satisfying \ref{hyp:g:nodes}--\ref{hyp:g:shape}. Then for every sufficiently small $ d > 0 $, there are \emph{exactly two}  increasing stationary fronts (up to a translation) $\bar{u},\hat{u}$ of \eqref{e:lde:Nagumo} which satisfy \eqref{e:stationary}--\eqref{e:connection}.

Furthermore, the front $\bar{u}$ corresponds to the symbolic sequence $\bar{s}=(\ldots000111\ldots)$ (or its shifts) and satisfies $\bar{u}_i\in[0,a_1)\cup(a_2,1]$ for all $i\in\Z$. Similarly, the front $\hat{u}$ corresponds to the symbolic sequence $\hat{s}=(\ldots000a11\ldots)$ (or its shifts) and satisfies $\hat{u}_i\in(a_1,a_2)$ for exactly one $i\in\Z$ with $\hat{s}_i=a$.
\end{thm}

In other words, this result shows that the stationary solutions $u$ corresponding to monotone symbolic sequences involving multiple letters $a$, e.g., 
$$
    (\ldots000aa111\ldots),\ (\ldots000aaa111\ldots),\ (\ldots000aaaa111\ldots), \text{ etc.,}
$$
are not monotone, see Figure~\ref{fig:connectionsx4}.

\begin{figure*}[t!]
        \subfloat[Stable monotone $\bar{u}$ corresponding to $\bar{s}=(\ldots0001111\ldots)$.]{%
            \includegraphics[width=.48\linewidth]{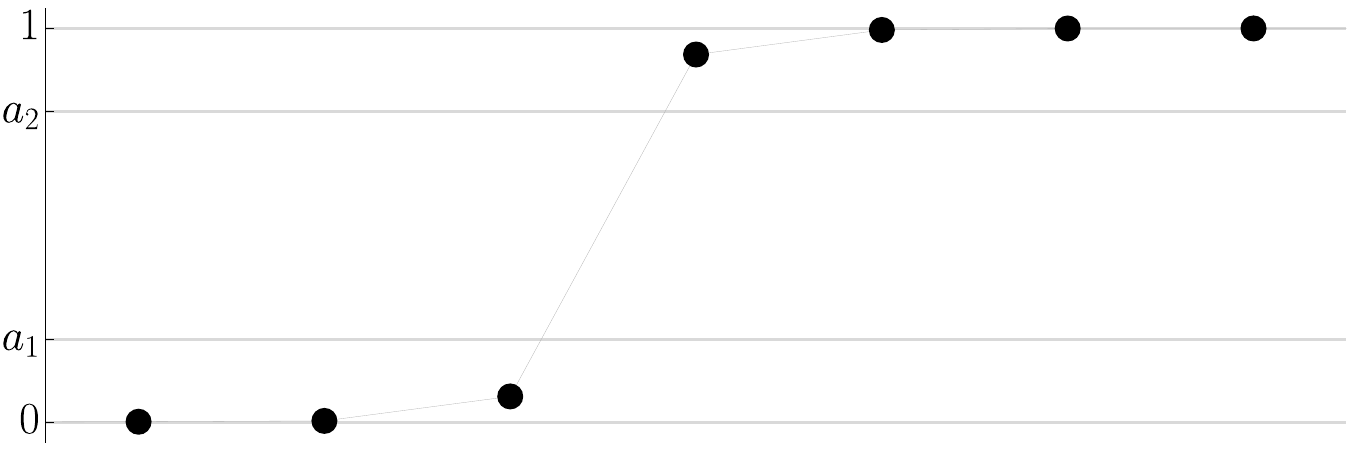}%
        }\hfill
        \subfloat[Unstable monotone $\hat{u}$ corresponding to $\hat{s}=(\ldots000a111\ldots)$.]{%
            \includegraphics[width=.48\linewidth]{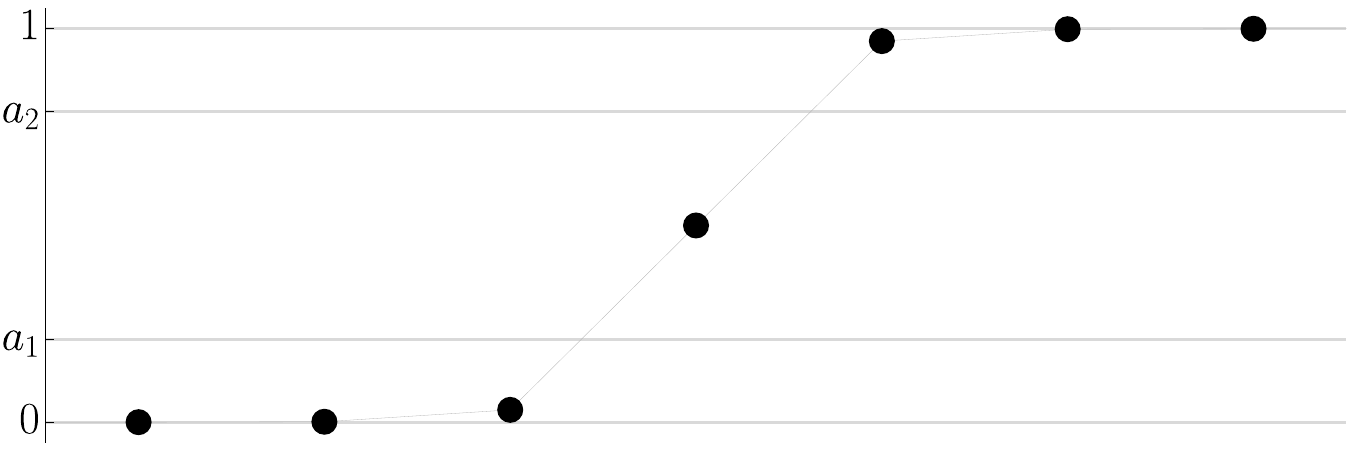}%
        }\\
        \subfloat[Unstable nonmonotone $\tilde{u}$ corresponding to $\tilde{s}=(\ldots000aa11\ldots)$.]{%
            \includegraphics[width=.48\linewidth]{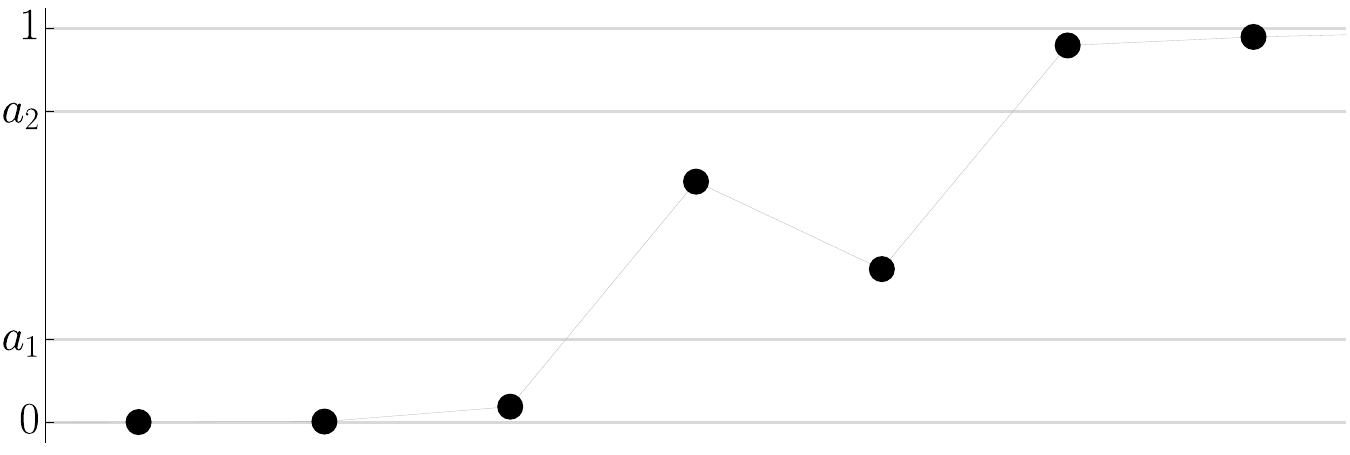}%
        }\hfill
        \subfloat[Unstable nonmonotone $\acute{u}$ corresponding to  $\acute{s}=(\ldots00aaa11\ldots)$.]{%
            \includegraphics[width=.48\linewidth]{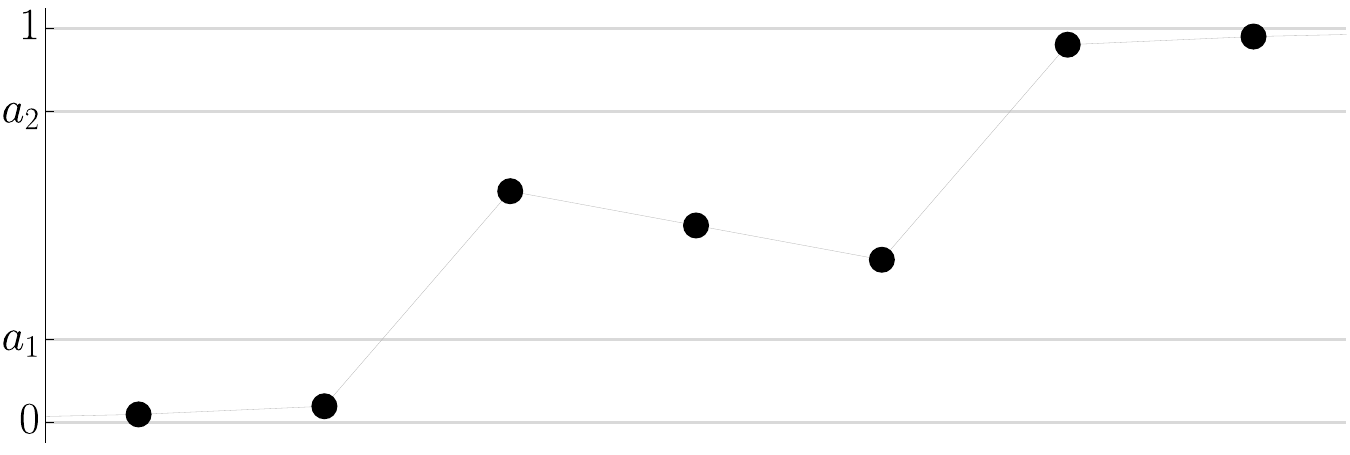}%
        }
        \caption{Illustration of Theorems~\ref{t:main-uniqueness}--\ref{t:main-stability}. Four stationary solutions $u\in[0,1]^\Z$ corresponding to four monotone symbolic sequences $s\in\Sigma_3=\{0,a,1\}^\Z$.}
        \label{fig:connectionsx4}
\end{figure*}

Besides the localization of stationary solutions and the number of monotone fronts, the a priori estimates allow us to characterize the stability of all stationary solutions. 

\begin{thm}[Stability]\label{t:main-stability}
    Let $g \in C^{2}([0,1]) $ be a bistable nonlinearity satisfying \ref{hyp:g:nodes}--\ref{hyp:g:shape}, $d>0$ be sufficiently small, and $u = (u_i)$, $i \in \mathbb{Z}$, be a stationary solution of the LDE~\eqref{e:lde:Nagumo} corresponding to $s\in\Sigma_3$. 
    \begin{enumerate}
        \item[(i)] If $u_i \in [0,a_1) \cup (a_2,1]$ (or equivalently $s_i\in\{0,1\}$) for all $i \in \mathbb{Z}$, then $u$ is locally asymptotically $\ell^2$-stable.
        \item[(ii)] If there exists an index $i_0 \in \mathbb{Z}$ such that $u_{i_0} \in (a_1,a_2)$ (or equivalently $s_{i_0}=a$), then $u$ is $\ell^2$-unstable.
    \end{enumerate}
\end{thm}

In other words, solutions $u$ corresponding to symbolic sequences $s\in\Sigma_2=\{0,1\}^\Z$ are locally asymptotically $\ell^2$-stable. Those corresponding to sequences $s\in\Sigma_3\setminus \Sigma_2$ (i.e., $s\in\Sigma_3$ with at least one symbol $a$) are $\ell^2$-unstable. We can then refine Theorem~\ref{t:main-twosol} and show that $\bar{u}$ is the unique locally asymptotically $\ell^2$-stable monotone front (up to a translation) and $\hat{u}$ is the unique $\ell^2$-unstable monotone front (up to a translation) of \eqref{e:lde:Nagumo} satisfying \eqref{e:connection}, see Figure~\ref{fig:connectionsx4}.



\paragraph{Symbolic dynamics perspective} Let us rephrase and illustrate our results from the point of view of the symbolic dynamics. We show that the map $\phi$ given by~\eqref{e:map:Keener} is homeomorphic to a variant of Smale's horseshoe map in which the square $Q=[0,1]^2$ is stretched horizontally and then bent twice, see Figure~\ref{fig:symbolic-caricature}. In this settings three vertical strips $V_i$ in the square are mapped onto horizontal strips $U_i$, $i\in\{0,a,1\}$. The repeated fractal-like forward and backward application of this simplified process creates a set of sectors identifiable with words $w\in\{0,a,1\}^{2n}$, $n\in\N$. Moreover, sectors described by letters from the two-letter alphabet $\{0,1\}$ correspond to stable solutions, see Figure~\ref{fig:symbolic-caricature}. Theorems~\ref{t:main-uniqueness}--\ref{t:main-stability} are proved by the analysis of the properties of the nonlinear version of this simplified process generated by the map $\phi$~\eqref{e:map:Keener} with sufficiently small $d$.

\paragraph{Paper structure}
In Section~\ref{sec:symbolic:Moser} we recall basic notions from the symbolic dynamics and formulate Moser's theorem. We then show topological conjugacy of stationary solutions of \eqref{e:lde:Nagumo} and three-letter biinfinite symbolic sequences, and prove Theorem~\ref{t:main-uniqueness} in Section~\ref{sec:uniqueness}. In Section~\ref{sec:monotonicity} we then use a priori estimates to prove Theorem~\ref{t:main-twosol} and show the existence of exactly two monotone fronts satisfying~\eqref{e:connection}. Finally, we analyze the stability of stationary solutions and prove Theorem~\ref{t:main-stability} in Section~\ref{sec:stability}. We conclude with a short discussion in Section~\ref{sec:discussion} where we weaken assumptions on the nonlinearity $g$ and extend Theorems~\ref{t:main-uniqueness}--\ref{t:main-twosol} to Frenkel-Kontorova type models.


\section{Symbolic dynamics and Moser's theorem}\label{sec:symbolic:Moser}
In  Section~\ref{sec:uniqueness} we are going to use Moser's theorem to show that the planar map $\phi$ defined by~\eqref{e:map:Keener} is topologically conjugate with the shift map $\sigma$ acting on the symbolic space of biinfinite sequences of three symbols $\Sigma_3$. In this section we introduce notation and formulate Moser's theorem.

\subsection{Symbolic space}\label{ss:symbolic:space} Throughout the paper we consider a three-letter \emph{alphabet} $\alphabet=\{0,a,1\}$. We kindly ask the reader to take into account the standard abuse of notation, we use $0,a,1$ both for the values in $\Real$ (roots of the bistability $g$, see \ref{hyp:g:nodes}) and symbols from $\alphabet$. We define the \emph{symbolic space of biinfinite sequences} on $\alphabet$ by
\[
\Sigma_3 = \alphabet^\Z = \{s=(\ldots s_{-2}s_{-1}\underline{s_0} s_1 s_2 \ldots): s_i\in\mathcal{A} \}.
\]
We underline the value corresponding to $0$-th entry of the biinfinite sequence $s\in\Sigma_3$. 
We are particularly interested in the (left) \emph{shift map} $\sigma:\Sigma_3\to\Sigma_3$ defined by
\[
	\sigma(s) = \sigma((\ldots s_{-2}s_{-1}\underline{s_0} s_1 s_2 \ldots)) = (\ldots s_{-1} s_0 \underline{s_1} s_2 s_3\ldots).
\]
The shift map $\sigma$ is a continuous and chaotic homeomorphism on $\Sigma_3$, see \cite{Devaney2022}.

\subsection{Moser's theorem}\label{ss:moser:theorem}
Smale's horseshoe \cite{Smale1967} is an elegant construction of the planar map on a unit square $Q=[0,1]^2$ which is topologically conjugate to the shift map $\sigma$. Moser's theorem \cite{Moser1975} provides a generalization for a wider class of planar maps $\phi:\Real^2\to \Real^2$. Before we formulate the theorem, let us introduce horizontal and vertical curves and strips in $Q$.

We say that $v(u)$ is a \emph{(Moser) horizontal curve} in $Q$ if:
\begin{enumerate}[label=(\roman*)]
    \item $0 \leq v(u) \leq 1$ for $0\leq u \leq 1$, and
    \item there exists $\mu\in(0,1)$ such that for all $u_1,u_2\in[0,1]$ we have $|v(u_1)-v(u_2)| \leq \mu |u_1-u_2|$.
\end{enumerate}

Similarly, $u(v)$ is a \emph{(Moser) vertical curve} in $Q$ if:
\begin{enumerate}[label=(\roman*)]
    \item $0 \leq u(v) \leq 1$ for $0\leq v \leq 1$, and
    \item there exists $\mu\in(0,1)$ such that for all $v_1,v_2\in[0,1]$ we have $|u(v_1)-u(v_2)| \leq \mu |v_1-v_2|$.
\end{enumerate}
The assumptions (ii), which are missing, e.g., in Keener \cite{Keener1987}, ensure that the vertical and horizontal curves have a unique intersection in the unit square $Q$, \cite{Moser1975}. Since we are interested in uniqueness and topological conjugacy we consider only Moser horizontal and vertical curves further. Specifically, if $v(u)=\bar{v}$ for all $u\in[0,1]$ (or $u(v)=\bar{u}$ for all $v\in[0,1]$) we speak about \emph{horizontal line} (or \emph{vertical line}).

For two disjoint horizontal curves $0\leq v_1(u) < v_2(u) \leq 1$ we call the set
$$
U = \{(u,v): 0\leq u \leq 1,\quad v_1(u) \leq v \leq v_2(u) \},
$$
a \emph{horizontal strip} in $Q$. Similarly, for two disjoint vertical curves $0\leq u_1(v) < u_2(v) \leq 1$ we call the set
$$
V = \{(u,v): u_1(v) \leq u \leq u_2(v),\quad 0\leq v \leq 1 \},
$$
a \emph{vertical strip} in $Q$. We define the \emph{width} of horizontal and vertical strips by
\[
\delta(U) = \max_{u\in[0,1]} (v_2(u)-v_1(u)),\quad \delta(V) = \max_{v\in[0,1]} (u_2(v)-u_1(v)).
\]

For our purposes we formulate Moser's theorem \cite[Theorem 3.1]{Moser1975} specifically for the three-letter alphabet $\alphabet$.

\begin{thm}[Moser]\label{t:Moser}
If $\phi:\Real^2\to \Real^2$ is a homeomorphism and there exists disjoint horizontal strips $U_i\subset Q$, $i\in \alphabet$, and disjoint vertical strips $V_i\subset Q$, $i\in \alphabet$, such that:
\begin{enumerate}[label=(\roman*)]
    \item $\phi(V_i)=U_i$ for all $i\in\alphabet$,
    \item the vertical boundaries of $V_i$ are mapped onto the vertical boundaries of $U_i$ for all $i\in\alphabet$ and the horizontal boundaries of $V_i$ are mapped onto the horizontal boundaries of $U_i$ for all $i\in\alphabet$,
    \item if $V_j$ is a vertical strip in $\bigcup_{i\in\alphabet} V_i$ then for all $i\in\alphabet$ the set
    $$
    V_{ij} := \phi^{-1} (V_j) \cap V_i
    $$
    is a vertical strip and there exists $\nu\in(0,1)$ such that
    $$
        \delta(V_{ij}) \leq \nu\cdot \delta(V_j),
    $$
    \item if $U_j$ is a horizontal strip in $\bigcup_{i\in\alphabet} U_i$ then for all $i\in\alphabet$ the set
    $$
    U_{ji} := \phi (U_j) \cap U_i
    $$
    is a horizontal strip and there exists $\nu\in(0,1)$ such that
    $$
        \delta(U_{ji}) \leq \nu\cdot \delta(U_j).
    $$
\end{enumerate}
Then the map $\phi$ possesses the shift $\sigma:\Sigma_3\to\Sigma_3$ as a subsystem, i.e., there exists a set $\Lambda\subset Q$ and a homeomorphism $\tau:\Sigma_3\to \Lambda$ such that $\phi(\tau(s)) = \tau(\sigma(s))$ for every $s\in\Sigma_3$.
\end{thm}

\begin{rmk}\label{r:Cantor_set}
    In other words Theorem~\ref{t:Moser} implies that there exists a Cantor set $\Lambda\subset Q$ such that for all $s\in\Sigma_3$ there exists a unique $(u,v)\in \Lambda$ with $\tau(s)=(u,v)$. The restriction of $\phi$ on $\Lambda$ is then a homeomorphism topologically conjugate with $\sigma$
    $$
    \begin{array}{ccc}
            \Sigma_3 & \xrightarrow{\sigma} & \Sigma_3\\
            \tau \downarrow & & \downarrow \tau \\
            \Lambda & \xrightarrow{\phi} & \Lambda \\
        \end{array}
    $$

    The contraction assumptions (iii) and (iv) together with the stricter definition of horizontal and vertical curves ensure that the map $\tau$ is a homeomorphism between $\Sigma_3$ and $\Lambda$. Without these we get, that $\tau$ maps biinfinite sequences from $\Sigma_3$ into $Q$, see \cite{Keener1987}.
\end{rmk}

\section{Topological conjugacy}\label{sec:uniqueness}
In this section we verify the assumptions of Theorem~\ref{t:Moser} for the map $\phi$ defined by~\eqref{e:map:Keener}. First, let us define an auxiliary function
\begin{equation}\label{e:h}
    h(u;a,d)=2u - \frac{1}{d}g(u;a),
\end{equation}
and rewrite concisely the map $\phi:\Real^2\to \Real^2$
\begin{equation}\label{e:phi}
    \phi(u,v)=\left(2u - v - \frac{1}{d}g(u;a) , u\right) = \left(h(u;a,d) - v , u\right).
\end{equation}
Its inverse satisfies
\begin{equation}\label{e:phiInv}
    \phi^{-1} (u, v) = \left(v, 2v - u - \frac{1}{d}g(v;a)\right)=(v,h(v;a,d)-u),
\end{equation}
which implies that $\phi$ is a homeomorphism on $\Real^2$.

We are interested in sufficiently small diffusion which implies the existence of three vertical strips required by Theorem~\ref{t:Moser}.

\begin{lem}[Properties of $h$]\label{l:roots}
    Let $g$ satisfy \ref{hyp:g:nodes}--\ref{hyp:g:shape}. Then there exists $d^*(a)$ such that for all $d<d^*(a)$ there are exactly\footnote{We use roman letters $\ru_i$ for roots to distinguish them from entries of the stationary solution $u=(u_i)$, $i\in\Z$, of the LDE~\eqref{e:lde:Nagumo}.}
    \begin{enumerate}[label=(\roman*)]
        \item three roots $0=\ru_0<\ru_5<\ru_6<1$ of $h(u;a,d)=0$,
        \item three roots $0<\ru_1<\ru_4<\ru_7<1$ of $h(u;a,d)=1$,
        \item three roots $0<\ru_2<\ru_3<\ru_8=1$ of $h(u;a,d)=2$.
    \end{enumerate}
    Moreover, the roots $\ru_i$, $i=0,1,\ldots,8,$ satisfy
    \[
    0 = \ru_0 < \ru_1 < \ldots < \ru_7 < \ru_8 =1,
    \]
    and the function $h(u;a,d)$ satisfies
    \[
        h'(u;a,d)\begin{cases}
            > 1, & u\in[0,\ru_2]\cup[\ru_6,1],\\
            < -1, & u\in[\ru_3,\ru_5].
        \end{cases}
    \]
\end{lem}
\begin{proof}
    The proof follows from the fact that for sufficiently small $d>0$ the assumptions \ref{hyp:g:nodes}--\ref{hyp:g:shape} ensure $h(a_1;a,d)>2$, $h(a_2;a,d)<0$ and 
    \[
    h'(u;a,d) = 2 - \frac{1}{d} g'(u;a).
    \qedhere
    \]
\end{proof}

\begin{figure}
    \centering
    \includegraphics[width=0.7\linewidth]{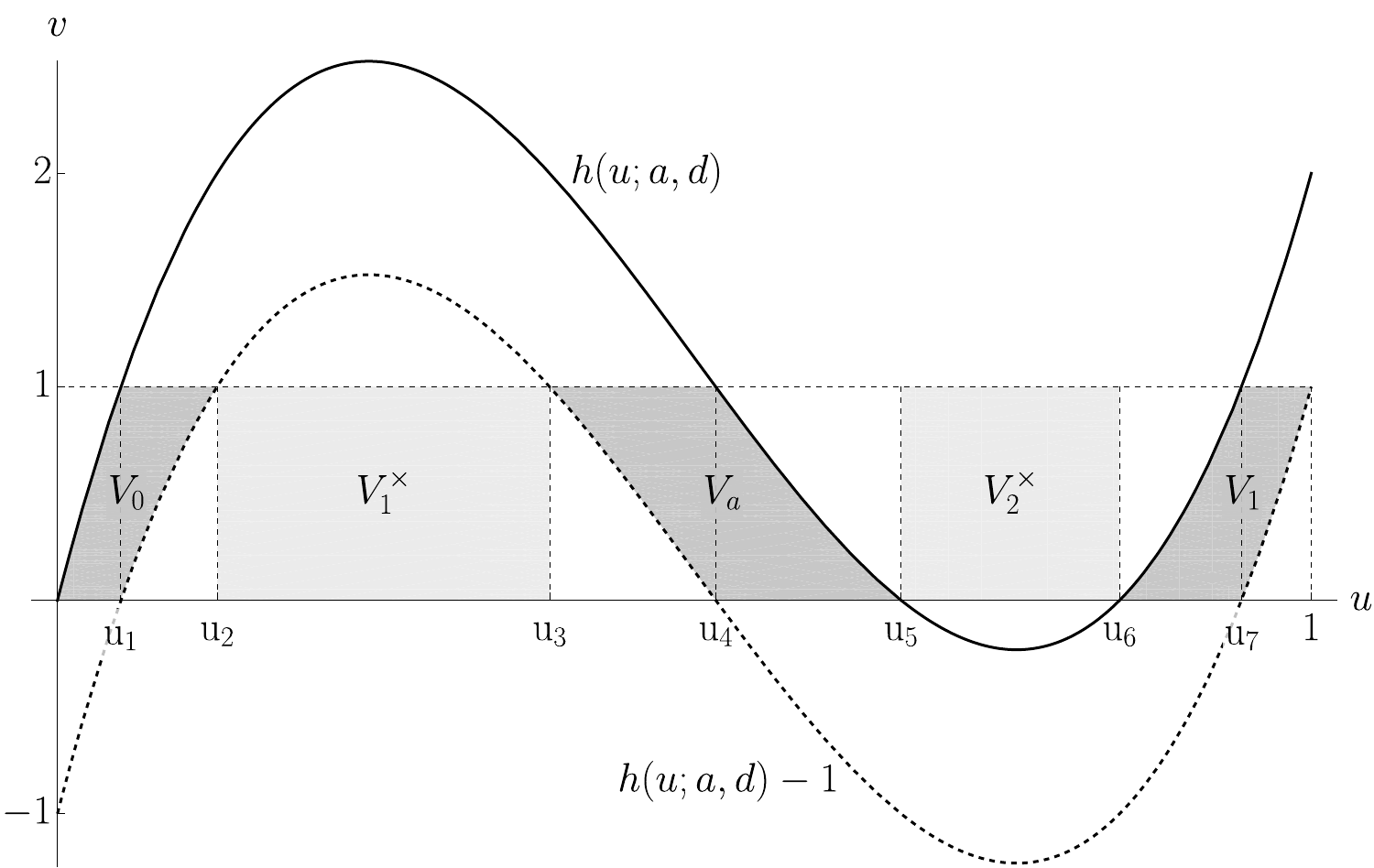}
    \caption{Roots $\ru_i$, $i=0,\ldots, 8$, from Lemma~\ref{l:roots}, vertical strips $V_i$, $i\in\alphabet$, defined by \eqref{e:curves:vetrical:strip:1}--\eqref{e:curves:vetrical:strip:12} and vertical strips $\fstrip_j$, $j\in\{1,2\},$ defined by \eqref{e:fstripL}--\eqref{e:fstripR}.}
    \label{fig:strips:Vi}
\end{figure}

The roots $\ru_i$, $i=0,1,\ldots,8,$ are depicted in Figure~\ref{fig:strips:Vi}. We now define the vertical strip $V_0$ whose boundaries are given by a pair of horizontal curves (the upper indices $\{d, u\}$ correspond to down and up) and a pair of vertical curves (the upper indices $\{l, r\}$ correspond to left and right)
\begin{align}
    \gamma_0^l&=\{(u,v):u\in[0,\ru_1],v=h(u;a,d)\},\label{e:curves:vetrical:strip:1}\\
    \gamma_0^d&=\{(u,v):u\in[0,\ru_1],v=0\},\\
    \gamma_0^r&=\{(u,v):u\in[\ru_1,\ru_2],v=h(u;a,d)-1\},\\
    \gamma_0^u&=\{(u,v):u\in[\ru_1,\ru_2],v=1\}.
\end{align}
The vertical strip $V_a$ is defined by the curves
\begin{align}
    \gamma_a^l&=\{(u,v):u\in[\ru_3,\ru_4],v=h(u;a,d)-1\},\\
    \gamma_a^u&=\{(u,v):u\in[\ru_3,\ru_4],v=1\},\\
    \gamma_a^r&=\{(u,v):u\in[\ru_4,\ru_5],v=h(u;a,d)\},\\
    \gamma_a^d&=\{(u,v):u\in[\ru_4,\ru_5],v=0\},
\end{align}
and the vertical strip $V_1$ by the curves
\begin{align}
    \gamma_1^l&=\{(u,v):u\in[\ru_6,\ru_7],v=h(u;a,d)\},\\
    \gamma_1^d&=\{(u,v):u\in[\ru_6,\ru_7],v=0\},\\
    \gamma_1^r&=\{(u,v):u\in[\ru_7,1],v=h(u;a,d)-1\},\\
    \gamma_1^u&=\{(u,v):u\in[\ru_7,1],v=1\}.\label{e:curves:vetrical:strip:12}
\end{align}
The horizontal strips $U_i$, $i\in\alphabet$, are defined by the curves $\Gamma_i^j$ which are inverse curves to $\gamma_i^j$,
\begin{equation}\label{e:curves:horizontal:strips}
    \Gamma_i^d := (\gamma_i^l)^{-1},\ \Gamma_i^l := (\gamma_i^d)^{-1},\ \Gamma_i^u := (\gamma_i^r)^{-1},\ \Gamma_i^r := (\gamma_i^u)^{-1},\quad i\in\alphabet.
\end{equation}

\begin{lem}[Mapping vertical strips onto horizontal strips]\label{l:vertical:strips:onto:horizontal}
    The horizontal strips $V_i$ and $U_i$, $i\in\alphabet$, defined by \eqref{e:curves:vetrical:strip:1}--\eqref{e:curves:horizontal:strips} satisfy:
    \begin{enumerate}[label=(\roman*)]
        \item $\phi(V_i)=U_i$ for each $i\in\alphabet$,
        \item the vertical boundaries of $V_i$ are mapped onto the vertical boundaries of $U_i$ for each $i\in\alphabet$ and the horizontal boundaries of $V_i$ are mapped onto the horizontal boundaries of $U_i$ for each $i\in\alphabet$.
    \end{enumerate}
\end{lem}
\begin{proof}
    The proof follows from the fact that $\Gamma_i^j$ are inverse curves of $\gamma_i^j$ for each $i\in\alphabet,$ $j\in\{l, d, r, u\}$. Moreover, the maps $\phi$ and $\phi^{-1}$ are symmetric in the following sense. If we define the reflection map $\rho:Q\to Q$ by $\rho(u,v)=(v,u)$ we have
    \[
            \phi^{-1}(u,v)=\rho(\phi(\rho(u,v))),
    \]
    since
    \[
        \rho(\phi(\rho(u,v))) = \rho(\phi(v,u)) = \rho(h(v;a,d)-u,u) = (u,h(v;a,d)-u) = \phi^{-1}(u,v).
    \qedhere
    \]
    
\end{proof}

In order to prove the contraction of vertical strips (assumption (iii) in Moser's Theorem~\ref{t:Moser}) we show that preimages of vertical lines are vertical shifts of the function $h(u;a,d)$.
\begin{lem}[Preimages of vertical lines]\label{l:preimages:vertical:lines}
    Let $\gamma^v$ be a vertical line
    $$
    \gamma^v= \{(u,v): u=\bar{u}, v\in[0,1] \}.
    $$
    Then its preimage is a curve
    $$
    \phi^{-1}(\gamma^v)=\{(u,v): u\in[0,1], v=h(u;a,d)-\bar{u} \}.
    $$
\end{lem}
\begin{proof}
    We apply directly the definition~\eqref{e:phiInv} of $\phi^{-1}$.
\end{proof}

Similarly, the images of horizontal lines are horizontal shifts of the curve parameterized by $(h(v;a,d), v)$, $v\in[0,1]$.

\begin{lem}[Images of horizontal lines]\label{l:images:horizontal:lines}
    Let $\gamma^h$ be a horizontal line
    $$
    \gamma^h= \{(u,v): u\in[0,1], v=\bar{v} \}.
    $$
    Then its image is a curve
    $$
    \phi(\gamma^h)= \{(u,v): u=h(v;a,d)-\bar{v}, v\in[0,1] \}.
    $$
\end{lem}

Specifically, if we consider the vertical lines defined by roots $\ru_i$, $i=0,\ldots, 8$, from Lemma~\ref{l:roots}, i.e.,
\[
\gamma_i^v = \{(u,v): u=\ru_i, v\in[0,1] \}, \quad i=0,\ldots, 8,
\]
then Lemma~\ref{l:preimages:vertical:lines} implies that their preimages $\phi^{-1}(\gamma_i^v)$ are graphs of functions
\begin{equation}\label{e:vi_curves}
    \rv_i(u)=h(u;a,d)-\ru_i,\quad u\in[0,1],\ i=0,\ldots, 8.
\end{equation}
This enables us to locate preimages of vertical strips $V_i$, $i\in\alphabet$. We also study preimages of vertical strips
\begin{align}
    \fstripL &:= \{(u,v):u\in[\ru_2,\ru_3], v\in[0,1] \},\label{e:fstripL}\\
    \fstripR &:= \{(u,v):u\in[\ru_5,\ru_7], v\in[0,1] \},\label{e:fstripR}
\end{align}
see Figure~\ref{fig:strips:Vi}. Note that $\left(Q\setminus \bigcup_{i\in\alphabet} V_i\right) \supset \left(\fstripL\cup\fstripR\right)$. 

\begin{lem}[Preimages of vertical strips]\label{l:preimages:vertical:strips}
Let $V_i$, $i\in\alphabet$, and $\fstrip_j$, $j\in\{1,2\}$, be vertical strips defined by \eqref{e:curves:vetrical:strip:1}--\eqref{e:curves:vetrical:strip:12} and \eqref{e:fstripL}--\eqref{e:fstripR}. Their preimages satisfy:
\begin{enumerate}[label=(\roman*)]
    \item $\phi^{-1}(V_0)\subset\{(u,v): u\in[0,1], v\in[\rv_0(u),\rv_2(u)] \}$,
    \item $\phi^{-1}(\fstripL) = \{(u,v): u\in[0,1], v\in[\rv_2(u),\rv_3(u)] \}$,
    \item $\phi^{-1}(V_a)\subset\{(u,v): u\in[0,1], v\in[\rv_3(u),\rv_5(u)] \}$,
    \item $\phi^{-1}(\fstripR) = \{(u,v): u\in[0,1], v\in[\rv_5(u),\rv_6(u)] \}$,
    \item $\phi^{-1}(V_1)\subset\{(u,v): u\in[0,1], v\in[\rv_6(u),\rv_8(u)] \}$.    
\end{enumerate}    
\end{lem}

\begin{proof}
    The proof follows from Lemma~\ref{l:preimages:vertical:lines}, the continuity of $\phi^{-1}$, definitions \eqref{e:fstripL}--\eqref{e:fstripR}, and the fact that
    \begin{align*}
        V_0 &\subset \{(u,v): u\in[\ru_0=0,\ru_2], v\in[0,1] \},\\
        V_a &\subset \{(u,v): u\in[\ru_3,\ru_5], v\in[0,1] \},\\
        V_1 &\subset \{(u,v): u\in[\ru_6,\ru_8=1], v\in[0,1] \}.
        \qedhere
    \end{align*}        
\end{proof}


\begin{figure}
    \centering
    \includegraphics[width=0.7\linewidth]{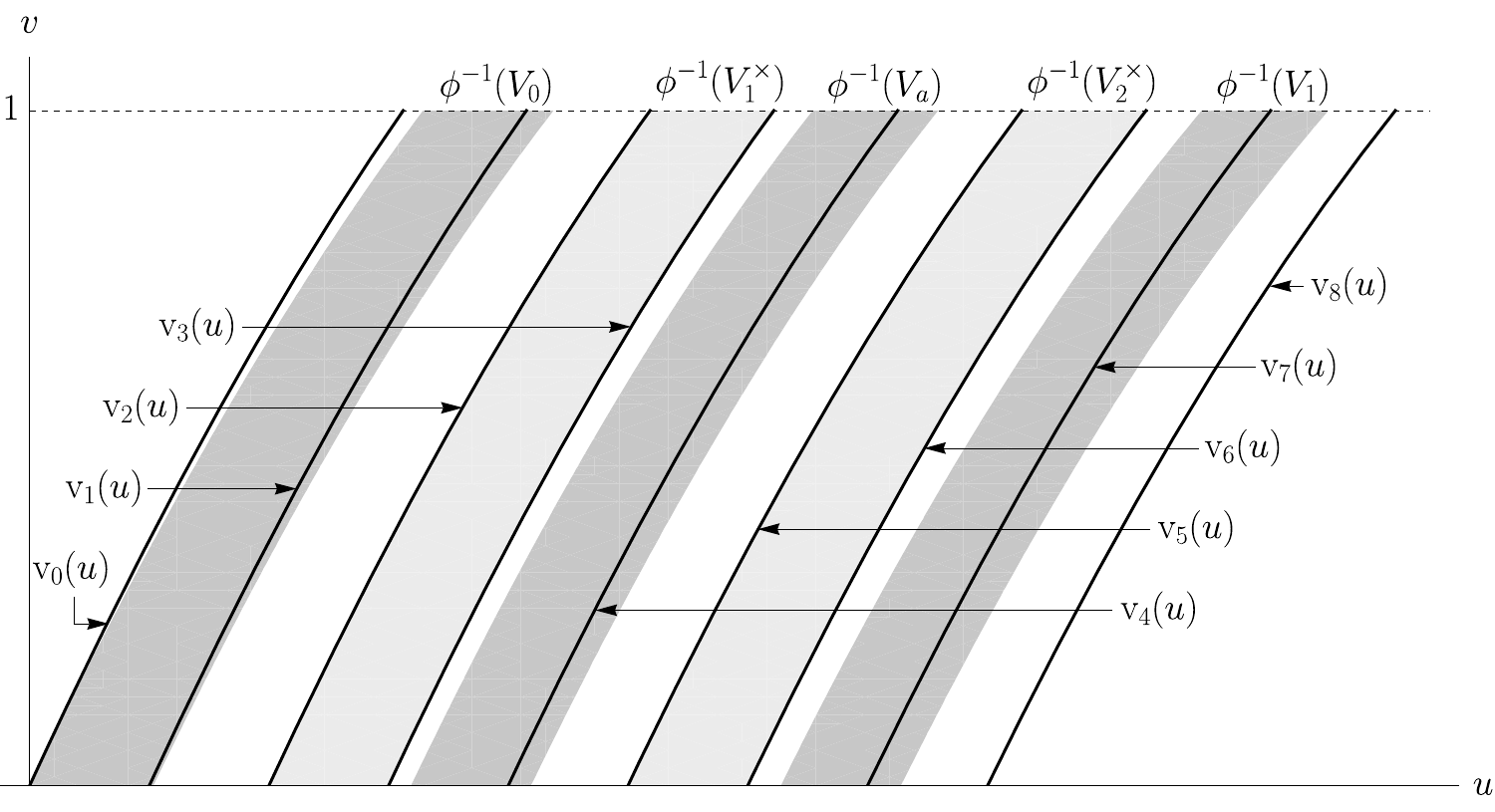}
    \caption{Illustration of Lemma~\ref{l:preimages:vertical:strips}. Preimages of vertical strips $V_i$, $i\in\alphabet$, and $\fstrip_j$, $j\in\{1,2\}$ (depicted only in the intersection with the strip $V_0$ for clarity).}
    \label{fig:preimages:vertical:strips}
\end{figure}

See Figure~\ref{fig:preimages:vertical:strips} for illustration of Lemma~\ref{l:preimages:vertical:strips}. Before we analyze widths of vertical strips $\phi^{-1} (V_j) \cap V_i$, let us state an auxiliary result.
\begin{lem}[Width of vertical strips defined by shifted functions]\label{l:width:vertical:strips:shifted:functions}
    Let $\varphi\in C([\alpha,\beta])$ be a strictly monotone function and $\psi(x)=\varphi(x)+c$, $c\neq 0$. Then there exists $\delta>0$ such that
    \[
        |\varphi^{-1}(y)-\psi^{-1}(y)| \geq \delta >0 \quad \text{for all} \quad y\in \rng(\varphi)\cap \rng(\psi).
    \]    
\end{lem}
\begin{proof}
    Let us define a continuous function $\Phi(y):=|\varphi^{-1}(y)-\psi^{-1}(y)|$, $y \in D = \rng(\varphi)\cap \rng(\psi)$. Weierstrass' theorem implies that there exists $y_0\in D$ such that $\delta=\Phi(y_0)=\min_{y\in D} \Phi(y)$. Assume by contradiction that $\delta=0$, i.e., $\varphi^{-1}(y_0)=\psi^{-1}(y_0)$, then
    \[
        y_0 = \psi(\psi^{-1}(y_0))=\psi(\varphi^{-1}(y_0))=\varphi(\varphi^{-1}(y_0))+c = y_0+c,
    \]
    a contradiction.
\end{proof}

Let us estimate the width of preimages of vertical strips $\fstrip_j$.
\begin{lem}[Width of preimages of $\fstripL, \fstripR$]\label{l:width:fstrips}
    There exists $\delta^\times>0$ such that for all $i\in\alphabet$ and all $j\in\{1,2\}$
    \[
        \delta(\phi^{-1}(\fstrip_j) \cap V_i) > \delta^\times.
    \]    
\end{lem}
\begin{proof}
   Lemma~\ref{l:preimages:vertical:strips}(ii) implies that $\phi^{-1}(\fstripL)\cap Q$ form three vertical strips  between functions $\rv_2(u)$ and $\rv_3(u)$. These functions are increasing on $[\ru_0=0,\ru_2]$ and $[\ru_6,\ru_8=1]$ and decreasing on $[\ru_3,\ru_5]$. Therefore, Lemma~\ref{l:width:vertical:strips:shifted:functions} implies that there exists $\delta_1^\times>0$ such that 
   \[
        \delta(\phi^{-1}(\fstripL)\cap V_i) \geq \delta_1^\times \quad \text{for all }i\in\alphabet.
   \]
   Similarly, Lemma~\ref{l:preimages:vertical:strips}(iv) implies the existence of $\delta_2^\times>0$ such that
   \[
        \delta(\phi^{-1}(\fstripR)\cap V_i) \geq \delta_2^\times \quad \text{for all }i\in\alphabet.
   \]
   We define $\delta^\times = \min\{\delta_1^\times,\delta_2^\times\}$ to conclude the proof.
\end{proof}

We can now use this result to estimate the width of preimages of $V_i$, $i\in\alphabet$.
\begin{lem}[Width of preimages of $V_i$]\label{l:width:preimages:Vi}
    There exists $\nu\in(0,1)$ such that the widths of vertical strips $V_{ji}:= \phi^{-1}(V_j)\cap V_i$, $i,j\in\alphabet$, satisfy
    \[
        \delta(V_{ij}) \leq \nu\cdot \delta(V_j)\quad \text{for all } i,j\in\alphabet.
    \]
    
\end{lem}
\begin{proof}
    Lemma~\ref{l:width:fstrips} implies 
    \[
        \delta(V_{ij}) \leq \delta(V_j)-2\delta^\times = \left(1-\frac{2\delta^\times}{\delta(V_j)}\right)\delta(V_j).
    \]
    We set $\nu=\max_{j\in\alphabet} \left(1-\frac{2\delta^\times}{\delta(V_j)}\right)<1$ to conclude the proof.
\end{proof}

The same argument can be used to obtain contracting property of images of horizontal strips.
\begin{lem}[Width of images of $U_i$]\label{l:width:images:Hi}
    There exists $\nu\in(0,1)$ such that the widths of horizontal strips $U_{ij}:= \phi(U_j)\cap U_i$, $i,j\in\alphabet$, satisfy
    \[
        \delta(U_{ji}) \leq \nu\cdot \delta(U_j)\quad \text{for all } i,j\in\alphabet.
    \]    
\end{lem}

We are now ready to prove Theorem~\ref{t:main-uniqueness}.
\begin{proof}[Proof of Theorem~\ref{t:main-uniqueness}]
We apply Moser's Theorem~\ref{t:Moser} by applying Lemma~\ref{l:vertical:strips:onto:horizontal} to verify its assumptions (i) and (ii). The assumption (iii) is implied by Lemma~\ref{l:width:preimages:Vi} and the assumption (iv) by Lemma~\ref{l:width:images:Hi}.
\end{proof}

\begin{rmk}[Cantor set]
    We can go even further, the contraction of the strips by the map $\phi$ (Lemmas~\ref{l:width:preimages:Vi} and~\ref{l:width:images:Hi}) ensures the existence of the set
    $$
        \Lambda_+=\{(u,v): \phi^n(u,v)\in Q \text{ for all } n\in\N\},
    $$
    which is an uncountable set of limit vertical curves. Analogously
    $$
        \Lambda_- =\{(u,v): \phi^{-n}(u,v)\in Q \text{ for all } n\in\N_0\},
    $$
    represent an uncountable set of limit horizontal curves. The intersection of these two sets
    $$
        \Lambda = \Lambda_+ \cap \Lambda_-
    $$
    is then invariant with respect to the arbitrary number of iteration of the map $\phi$ and represent the Cantor set from Remark~\ref{r:Cantor_set}. To rephrase, Theorem~\ref{t:main-twosol} states that the restriction  $\phi_\Lambda:\Lambda\to\Lambda$ of $\phi$ defined by~\eqref{e:map:Keener} on $\Lambda$ is topologically conjugate with the shift map $\sigma$ on $\Sigma_3$. In other words, for every $(u_0,v_0)\in\Lambda$ there is a unique $s=\tau^{-1}(u_0,v_0)\in\Sigma_3$ such that
    \begin{equation}\label{e:Usi}
        (u_i,v_i) = \phi_\Lambda^i(u_0,v_0)=\tau(\sigma^i(s))\in U_{s_i}\quad \text{for all } i\in\Z.
    \end{equation}
    There is a natural symmetry between the sets of vertical and horizontal curves $\Lambda_+, \Lambda_-$, which follows from the symmetry of vertical and horizontal strips, Lemma~\ref{l:vertical:strips:onto:horizontal}. From its proof we have
    $$
        \phi^{-n}(u,v)=\rho(\phi^n(\rho(u,v))),
    $$
    which implies
    \begin{align*}
        \Lambda_- =& \{(u,v): \phi^{n}(\rho(u,v))\in \rho(Q) \text{ for all } n\in\N_0\}\\
        =&\{(u,v): \phi^{n}(v,u)\in Q \text{ for all } n\in\N_0\} = \Lambda_+^{-1}.
    \end{align*} 
\end{rmk}


\section{Exactly two stationary monotone fronts}\label{sec:monotonicity}
In order to prove the monotonicity of biinfinite sequences $u=(u_i)$, $i \in \mathbb{Z}$, which solve~\eqref{e:stationary} we locate sectors in the unit square $Q$ which arise as intersection of horizontal and vertical strips defined by~\eqref{e:curves:vetrical:strip:1}--\eqref{e:curves:horizontal:strips}. 
\begin{defn}[Sectors]
    Let $d>0$ be sufficiently small, $n\in\N$, and $w\in\alphabet^{2n}$ be an (anchored) word of length $2n$ starting at the coordinate $-n+1$
    \[
        w=\left(w_{-n+1}w_{-n+2}\ldots w_{-1}\underline{w_0}w_1 \ldots w_{n-1}w_n\right).
    \]
    We call the set
    \[
        \sector_w = \{(u,v)\in Q: \phi^j(u,v)\in U_{w_j}, j\in\{-n+1,\ldots, n\}\}
    \]
    a \emph{sector} in $Q$ of order $n$.
\end{defn}

\begin{figure}
    \centering
    \includegraphics[width=0.4\linewidth]{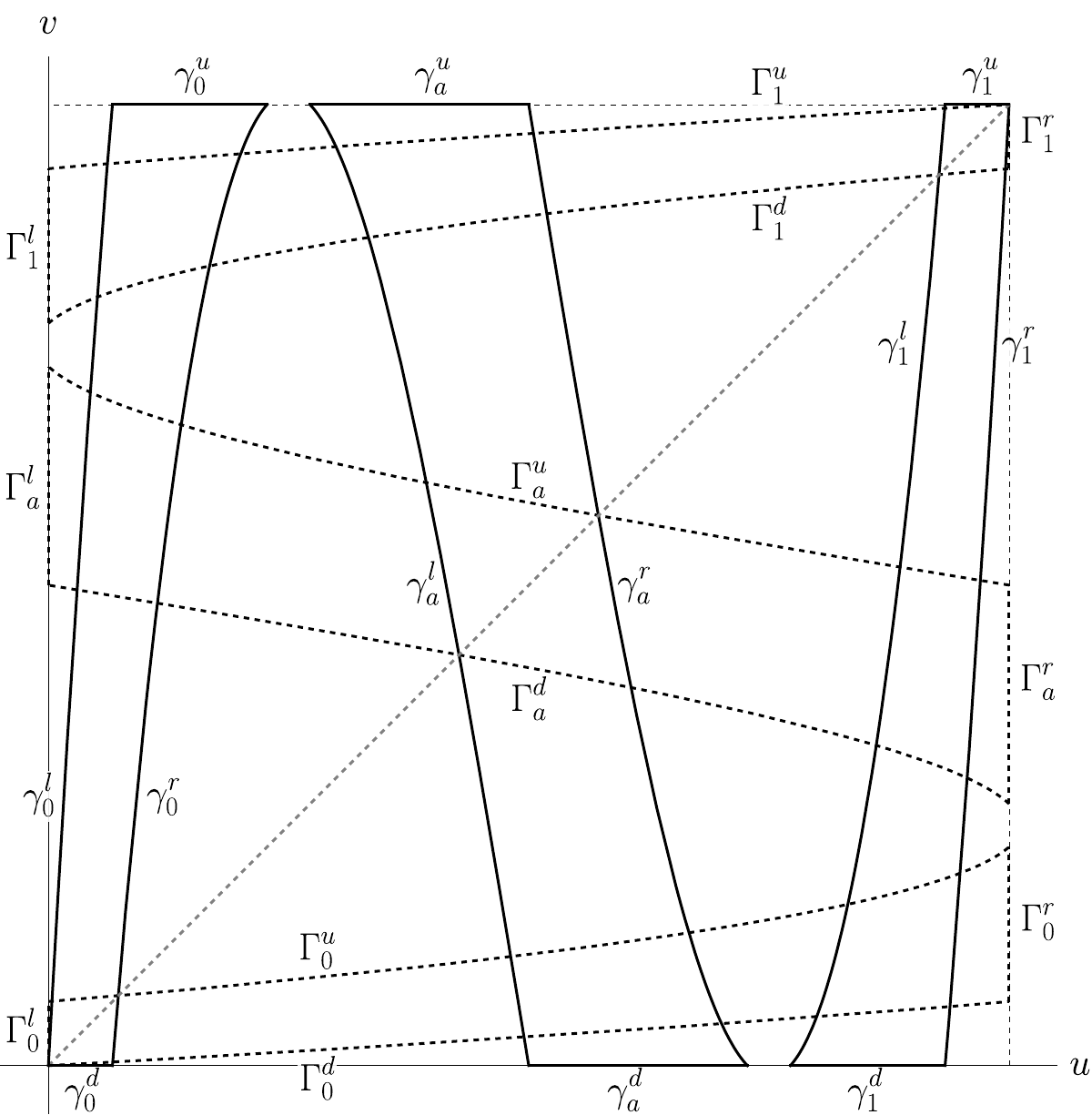}\
    \includegraphics[width=0.4\linewidth]{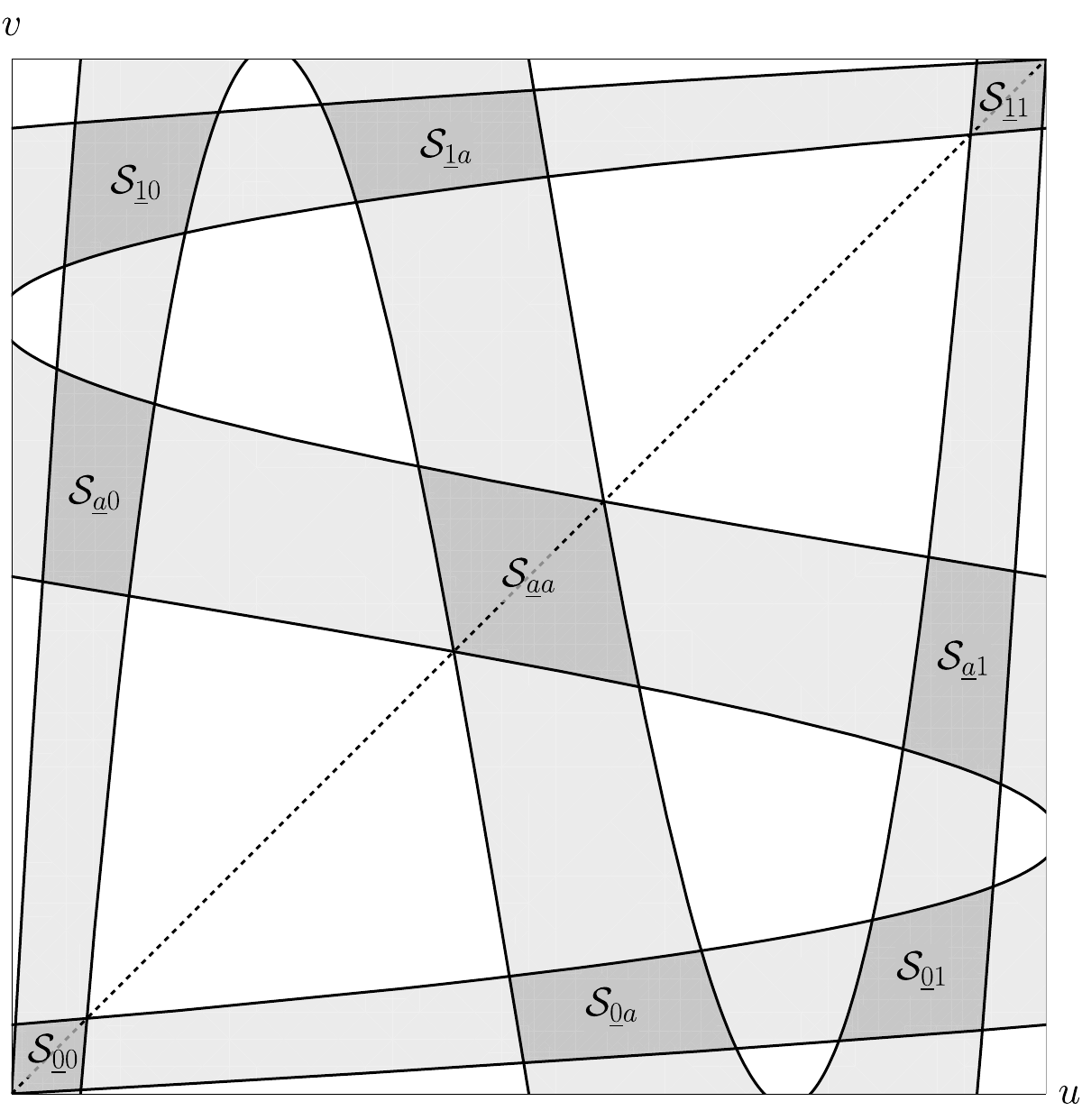}
    \caption{Illustration of Lemma~\ref{l:localization:1sector}. Left panel shows boundary curves  given \eqref{e:curves:vetrical:strip:1}--\eqref{e:curves:vetrical:strip:12} and define vertical strips $V_0, V_a, V_1$ (thick boundaries) and horizontal strips $U_0, U_a, U_1$ (dashed boundaries), see also Figure~\ref{fig:strips:Vi} The right panel depicts nine sectors of order 1 given by the intersections of these vertical and horizontal strips \eqref{e:sectors}.}
    \label{fig:1sectors}
\end{figure}

\begin{rmk}[Number of sectors]
    Let us recall that the sets
    \[
        \{(u,v)\in Q: \phi^j(u,v)\in U_{i} \},\quad i\in\alphabet, 
    \]
    form vertical strips if $j\in\N$ and horizontal strips if $j\in\Z\setminus\N$ in $Q$. Consequently, there are $3^{2n}$ sectors of order $n$, see Figure~\ref{fig:1sectors}. Sectors of  order $n$ arise as intersection of $3^n$ horizontal strips and $3^n$ vertical strips. For example, 9 sectors of order $1$ satisfy
    \begin{equation}\label{e:sectors}
        \sector_{\underline{i}j} = U_i\cap V_j,\quad i,j\in\alphabet,
    \end{equation}        
    where $U_i$ and $V_j$ are defined by \eqref{e:curves:vetrical:strip:1}--\eqref{e:curves:horizontal:strips}. Similarly 81 sectors of order 2 satisfy
    \[
        \sector_{i\underline{j}kl} = U_{ij}\cap V_{kl},\quad i,j,k,l\in\alphabet,
    \]
    where $U_{ij}$ and $V_{kl}$ are defined by Theorem~\ref{t:Moser}. We can generalize the definitions of vertical and horizontal strips to get arbitrary vertical strips recursively
    \[
        V_{j_1 j_2\ldots j_n} = \phi^{-1}(V_{j_2\ldots j_{n-1}}) \cap V_{j_1},\quad j_i\in\alphabet,
    \]
    and horizontal strips
    \[
        U_{j_n \ldots j_2 j_1} = \phi(U_{j_{n}\ldots j_2}) \cap U_{j_1},\quad j_i\in\alphabet.
    \]
    Sectors $\sector_w$ from~\eqref{e:sectors} can then be written as (see the schematic illustration in Figure~\ref{fig:symbolic-caricature})
    \[
        \sector_w = U_{w_{-n+1}\ldots w_{0}}\cap V_{w_{1}\ldots w_{n}}.
    \]
\end{rmk}

\begin{lem}[A priori estimate]\label{l:a_priori}
    Let $d>0$ be sufficiently small, $s\in\Sigma_3$, and $u=(\ldots, u_{-1},u_0,u_1,\ldots)$ be a solution of \eqref{e:stationary} that satisfies \eqref{e:connection} and $\tau(\sigma^i(s))=(u_i,u_{i-1})$ for all $i\in\mathbb{Z}$, where $\tau:\Sigma_3\to\Lambda$ is a homeomorphism from Theorem~\ref{t:main-uniqueness}. Then for all $i\in\Z$ and all $n\in\N$ we have
    $$
        (u_i,u_{i-1})\in\sector_{s_{i-n+1}\ldots s_{i-1}\underline{s_i},s_{i+1}\ldots s_{i+n}}.
    $$    
\end{lem}
\begin{proof}
    We have from~\eqref{e:Usi} that $\tau(\sigma^i(s))=(u_i,u_{i-1})\in U_{s_i}$. Therefore,
    \begin{align*}
        (u_0,u_{-1}) &\in U_{s_0},\\
        (u_0,u_{-1}) &\in U_{s_{-1}s_0} = \phi(U_{s_{-1}})\cap U_{s_0}, \text{ since } (u_{-1},u_{-2})\in U_{s_{-1}},\\
        \vdots \\
        (u_0,u_{-1}) &\in U_{s_{-n+1}\ldots s_{-1}s_0} = \phi(U_{s_{-n+1}\ldots s_{-1}})\cap U_{s_0}, \text{ since } (u_{-1},u_{-2})\in U_{s_{-n+1}\ldots s_{-1}}.
    \end{align*}
    Analogously,
    \begin{align*}
        (u_0,u_{-1})&\in V_{s_1},\text{ since } (u_1,u_{0})\in U_{s_1}, \text{ and } \phi(U_{s_1})=V_{s_1},\\
        (u_0,u_{-1}) &\in V_{s_{1}s_2} = \phi(V_{s_{2}})\cap V_{s_1}, \text{ since } (u_{1},u_{0})\in V_{s_{2}},\\
        \vdots \\
        (u_0,u_{-1}) &\in V_{s_{1}\ldots s_n} = \phi(V_{s_2\ldots s_{n}})\cap V_{s_1}, \text{ since } (u_{1},u_{0})\in V_{s_2\ldots s_{n}}.
        \qedhere
    \end{align*}
\end{proof}


First, we localize 6 sectors of order $1$, see Figure~\ref{fig:1sectors}.

\begin{lem}[Properties of sectors of order $1$]\label{l:localization:1sector}
    Let $d>0$ be sufficiently small. If $(u,v)\in\sector_{\underline{a}0}\cup\sector_{\underline{1}0}\cup\sector_{\underline{1}a}$ then $v>u$. Similarly, if $(u,v)\in\sector_{\underline{0}a}\cup\sector_{\underline{0}1}\cup\sector_{\underline{a}1}$ then $u>v$.    
\end{lem}
\begin{proof}
    Let us prove the former statement, the latter is proven accordingly. We observe that the lower-left and upper-right corners of the sector $\sector_{\underline{0}0}$ lie on the line $u=v$ since the curve $\gamma_0^l$ is the inverse of $\Gamma_0^d$, and the curve $\gamma_0^r$ is the inverse of $\Gamma_0^u$, see \eqref{e:curves:horizontal:strips}. It is now enough to observe that the curves $\Gamma_a^d$, $\Gamma_a^u$, $\Gamma_1^d$, $\Gamma_1^u$ lie above $\Gamma_0^u$, see \eqref{e:curves:horizontal:strips}, to observe that sectors $\sector_{\underline{a}0}$, $\sector_{\underline{1}0}$ lie above the the line $u=v$.

    In the same spirit, we observe that the lower-left and upper-right corners of the sector $\sector_{\underline{a}a}$ lie on the line $u=v$ since the curve $\gamma_a^l$ is the inverse of $\Gamma_a^d$, and the curve $\gamma_a^r$ is the inverse of $\Gamma_a^u$. Consequently, also the sector $\sector_{\underline{1}a}$ lies above the line $u=v$.
\end{proof}

\begin{figure}
    \centering
    \includegraphics[width=0.45\linewidth]{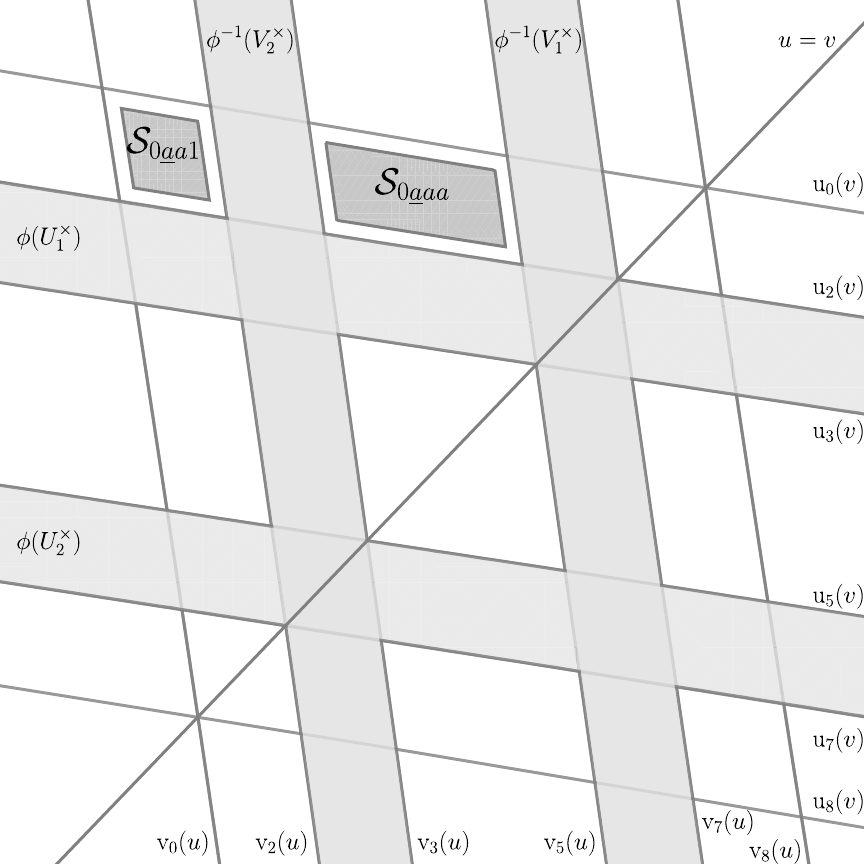}\ \includegraphics[width=0.45\linewidth]{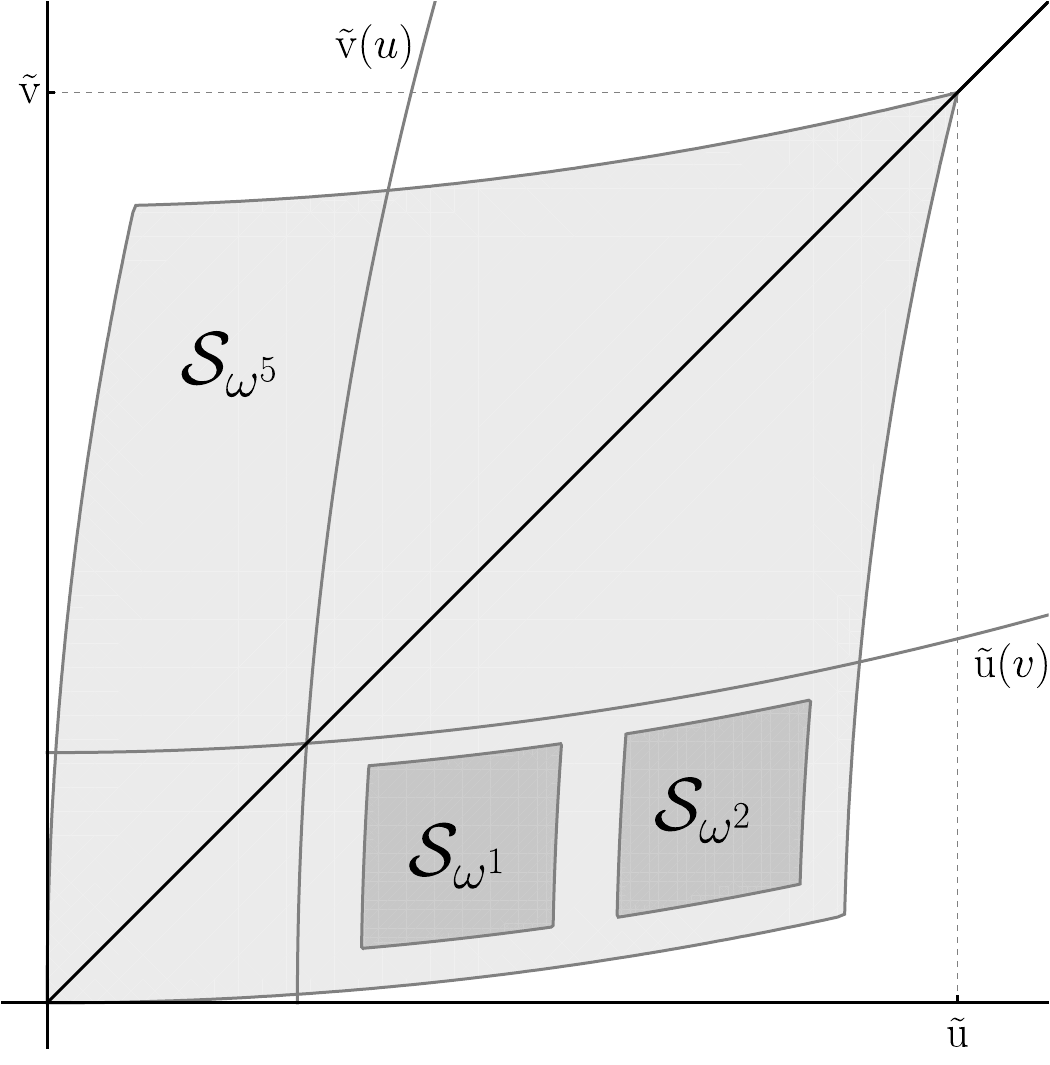}
    \caption{Illustration of Lemma~\ref{l:localization:2sector} (left panel) and Lemma~\ref{l:localization:nsector} (right panel). The left panel is a zoomed-in detail within the sector $\sector_{\underline{a}a}$, see Figure~\ref{fig:1sectors}. Similarly, the right panel highlights the situation in the neighborhood of origin, the sector $\sector_{\omega^5}$, $w^5=0\ldots\underline{0}\ldots 0$, is an arbitrary sector of order $n-1$, $n\in\Z$.}
    \label{fig:nsectors}
\end{figure}

Next, we localize 2 sectors of order $2$, see the left panel of Figure~\ref{fig:nsectors}.

\begin{lem}[Properties of sectors of order $2$]\label{l:localization:2sector}
    Let $d>0$ be sufficiently small. If $(u,v)\in\sector_{0\underline{a}a1}\cup\sector_{0\underline{a}aa}$  then $v>u$. 
\end{lem}
\begin{proof}
    Let the functions $\rv_i(u)$, $i=0,1,\ldots,8$, be defined by \eqref{e:vi_curves}. Their restrictions $\rv_i(u)|_{[\ru_3,\ru_5]}$ are decreasing on $[\ru_3,\ru_5]$ and have decreasing inverses $\ru_i(v)$. The curves $\ru_2(v)$ and $\rv_2(u)$ intersect on the line $u=v$, so do the curves $\ru_3(v)$ and $\rv_3(u)$. The statement follows ideas from the proof of Lemma~\ref{l:preimages:vertical:strips} and the fact that $\sector_{0\underline{a}aa}$ lies above $\ru_2(v)$ and left from $\rv_3(u)$. Similarly, the sector $\sector_{0\underline{a}a1}$ lies above $\ru_2(v)$ and left from $\rv_6(u)$, see Figure~\ref{fig:nsectors}.
\end{proof}

Finally, we localize sectors of order $n$, see the right panel of Figure~\ref{fig:nsectors}.

\begin{lem}[Properties of sectors of order $n$]\label{l:localization:nsector}
    Let $d>0$ be sufficiently small and 
    $$
        w^1=\underbrace{0\ldots\underline{0}\ldots0}_{(2n-1)\rm{-times}}\!a,\quad w^2=\underbrace{0\ldots\underline{0}\ldots0}_{(2n-1)\rm{-times}}\!1,\quad  w^3=0\!\underbrace{1\ldots\underline{1}\ldots1}_{(2n-1)\rm{-times}}, \text{ and } w^4=a\!\underbrace{1\ldots\underline{1}\ldots1}_{(2n-1)\rm{-times}}.
    $$
    Then for all $(u,v)\in\sector_{w^i}$, $i=1,2,3,4$, the inequality $u>v$ holds.
\end{lem}
\begin{proof}
    Let us consider the constant word $w^5=0\ldots\underline{0}\ldots 0$ of length $2(n-1)$. The sector $\sector_{w^5}$ lies in the square $[0,\tilde{\ru}]^2$. The function $\tilde{\rv}(u)=h(u;a,d)-\tilde{\ru}$ and its inverse $\tilde{\ru}(v)$ is increasing on $[0,\tilde{\ru}]$. The statement for $w^1$ and $w^2$ follows from the fact $\sector_{w^1}$ and $\sector_{w^2}$ lie to the right from $\tilde{\rv}(u)$ and below $\tilde{\ru}(v)$, see Figure~\ref{fig:nsectors}. The proof for $w^3$ and $w^4$ is similar, we only study the parts of the sector $\sector_{w^6}$ where $w^6=1\ldots\underline{1}\ldots 1$ is a constant word of length $2(n-1)$.
\end{proof}

We introduce a natural ordering $\lhd$ on the alphabet $\alphabet$ by $0\lhd a \lhd 1$. A biinfinite symbolic sequence $s\in\Sigma_3$ is nondecreasing if $s_i\unlhd s_{i+1}$ for all $i\in\Z$. Orderings $\rhd$ and $\unrhd$ and nonincreasing symbolic sequences are defined accordingly.

The fact that the corresponding symbolic sequence $s\in\Sigma_3$ is nondecreasing is necessary for the existence of increasing fronts of \eqref{e:stationary}.
\begin{lem}[Increasing symbolic sequences.]\label{l:increasing:symbols}
    If $u$ is a solution of \eqref{e:stationary}  satisfying \eqref{e:connection} then the unique symbolic sequence $s\in\Sigma_3$ with $\tau(\sigma^i(s))=(u_i,u_{i-1})$ for all $i\in\Z$ is nondecreasing.
\end{lem}
\begin{proof}
    Consider $u$ a solution of \eqref{e:stationary} satisfying the monotonicity and limit conditions \eqref{e:connection}. The uniqueness of $s\in\Sigma_3$ follows from Theorem~\ref{t:main-uniqueness}. Assume by contradiction that $s$ is not an nondecreasing sequence. Then there exists $i\in \Z$ such that $s_i\rhd s_{i+1}$. Then Lemma~\ref{l:a_priori} implies that $(u_i,u_{i-1})\in\sector_{\underline{s_i}s_{i+1}}$. However, $s_i\rhd s_{i+1}$ implies that $(u_i,u_{i-1})\in\sector_{\underline{a}0}\cup\sector_{\underline{1}0}\cup\sector_{\underline{1}a}$ and Lemma~\ref{l:localization:1sector} leads to $u_{i-1}>u_i$, a contradiction with the monotonicity condition \eqref{e:connection}.
\end{proof}

The key message of this section and Theorem~\ref{t:main-twosol} lies in the fact that the monotonicity of the symbolic sequence $s$ is indeed only a necessary condition for the monotonicity of the solution $u$.

\begin{rmk}[Notation of biinfinite symbolic sequences]
    Lemma~\ref{l:increasing:symbols} implies that $s$ (corresponding to the monotone solution $u$) is an nondecreasing biinfinite symbolic sequence. This implies that it must have the form
    $$
        s = \overline{0}(a)_k\overline{1},
    $$
    by which the notation indicates that the infinite sequence of zeroes $\overline{0}$ is followed by a constant word $w=a\ldots a$ of length $k\in\N_0$, and the infinite sequence of ones $\overline{1}$. 
\end{rmk}

\begin{lem}[At most one $a$ in the symbolic sequence]\label{l:at_most_one_a}
    If $u$ is a solution of \eqref{e:stationary} satisfying \eqref{e:connection} then the unique symbolic sequence $s\in\Sigma_3$ satisfying $\tau(\sigma^i(s))=(u_i,u_{i-1})$ for all $i\in\Z$ contains at most one symbol $a$.
\end{lem}
\begin{proof}
    Lemma~\ref{l:increasing:symbols} implies that $s$ is nondecreasing and has the form $s = \overline{0}(a)_k\overline{1}$ for some $k\in\N_0$. Assume by contradiction that $k>1$. Then, there exists $i\in\mathbb{Z}$ such that
    $$
        (u_i,u_{i-1})\in\sector_{0\underline{a}a1},\text{ or } (u_i,u_{i-1})\in\sector_{0\underline{a}aa}.
    $$
    Lemma~\ref{l:localization:2sector} leads to $u_{i-1}>u_i$, a contradiction with the monotonicity condition \eqref{e:connection}.
\end{proof}

We are now ready to prove Theorem~\ref{t:main-twosol}.
\begin{proof}[Proof of Theorem~\ref{t:main-twosol}]
Lemma~\ref{l:at_most_one_a} implies that $s\in\Sigma_3$ corresponding to the solution of $u$ of \eqref{e:stationary} satisfying \eqref{e:connection} must be of the form $s = \overline{0}\underline{0}\overline{1}$ or $r = \overline{0}\underline{a}\overline{1}$ (up to a shift). Let us prove that both these solutions are indeed increasing.

Let $u$ be a solution of \eqref{e:stationary} such that $\tau(\sigma^i(s))=(u_i,u_{i-1})$ for all $i\in\Z$. Then we have (Lemma~\ref{l:a_priori})
$$
    \ldots, (u_{-1},u_{-2})\in \sector_{00\underline{0}001},\ (u_0,u_{-1})\in \sector_{0\underline{0}01},\ (u_1,u_0)\in \sector_{\underline{0}1},\ (u_2,u_1)\in \sector_{0\underline{1}11},\ (u_3,u_2)\in \sector_{01\underline{1}111},  \ldots
$$
Lemma~\ref{l:localization:2sector}  (applied once) and Lemma~\ref{l:localization:nsector} (applied in the remaining cases) then imply $u_{i-1}<u_i$ for all $i\in\Z$.

Similarly, let $v$ be a solution of \eqref{e:stationary} such that $\tau(\sigma^i(r))=(v_i,v_{i-1})$ for all $i\in\Z$. Then the a priori estimate Lemma~\ref{l:a_priori} yields
$$
    \ldots, (v_{-2},v_{-3})\in \sector_{00\underline{0}00a},\ (v_{-1},v_{-2})\in \sector_{0\underline{0}0a},\ (v_0,v_{-1})\in\sector_{\underline{0}a}, (v_1,v_0)\in \sector_{\underline{a}1},\ (v_2,v_1)\in \sector_{a\underline{1}11},\ (v_3,v_2)\in \sector_{a1\underline{1}111},  \ldots
$$
Lemma~\ref{l:localization:2sector}  (applied twice in this case) and Lemma~\ref{l:localization:nsector} (applied in the remaining cases) then imply $v_{i-1}<v_i$ for all $i\in\Z$.

To conclude the proof, we observe that $u_i,v_i\to 0$ as $i\to-\infty$, which follows from the fact that $(0,0)\in\sector_{0\ldots\underline{0}0\ldots 0}$ for any sector of order $n\in\N$, the contraction of the width of these sectors with $\nu\in(0,1)$ (Lemma~\ref{l:width:preimages:Vi}), and the fact that
\[
    (u_{-i-1},u_{-i-2}), (v_{-i},v_{-i-1})\in \sector_{\underbrace{0\ldots\underline{0}0\ldots0}_{2i\rm{-times}}}.
\]
Similarly, we obtain $u_i,v_i\to 1$ as $i\to\infty$.
\end{proof}

\begin{rmk}
    Let us compare our results to Keener's \cite{Keener1987}. The inclusion of the full alphabet $\alphabet$ (in contrast to the reduced $\mathcal{A}_2=\{0,1\}\subset\alphabet$) has allowed us to show the existence of topological conjugacy between $\phi$ and $\sigma$, see Theorem~\ref{t:main-twosol} and Remark~\ref{r:Cantor_set}. In other words, all stationary solutions $u$ of  \eqref{e:stationary} satisfying \eqref{e:connection} can be represented via the symbolic sequence $s\in\Sigma_3$.
    
    Regarding monotonicity, considering the reduced alphabet $\mathcal{A}_2=\{0,1\}$ allows to formulate Lemma~\ref{l:increasing:symbols} as a necessary and sufficient condition, i.e., all nondecreasing symbolic sequences $s\in\Sigma_2=\mathcal{A}_2^\Z$ yield increasing $u$'s, see the first part of the Proof of Theorem~\ref{t:main-twosol}. The key difference is Lemma~\ref{l:localization:2sector} which shows reversed inequalities for solutions corresponding to symbolic sequences with blocks $(a)_k$, $k\in\N$, $k>1$. The intuition behind this lies in the opposite orientation of horizontal and vertical strips corresponding to the letter $a$, see Lemma~\ref{l:localization:2sector}. Whereas the vertical strips $V_0$, $V_1$ ``increase'' in $u$, the vertical strip $V_a$ ``decreases'', etc., see Figure~\ref{fig:strips:Vi}.    
\end{rmk}

\section{Stability}\label{sec:stability}
In this section we prove Theorem~\ref{t:main-stability}  and show that stationary solutions which are topologically conjugate to biinfinite sequences in $\Sigma_2=\{0,1\}^\Z$ are locally asymptotically $ \ell^{2} $-stable whereas solutions topologically conjugate to sequences in $\Sigma_3\setminus\Sigma_2$ are $ \ell^{2} $-unstable.

Let $ u = (u_{i}) \in \ell^{\infty}(\mathbb{Z}) $ be a stationary solution of~\eqref{e:lde:Nagumo} such that $ u_{i} \in [0,1] $ for all $ i \in \mathbb{Z} $. We study the $ \ell^{2} $-stability of $ u $ by tracking solutions $ \chi(t) $ of~\eqref{e:lde:Nagumo} starting at initial conditions which are $ \ell^{2} $-perturbations of $ u $, i.e.,
\begin{equation}\label{e:stab-perturbation-IC}
\chi(0) = \chi^{0} = u + \psi^{0}, \quad \psi^{0} \in \ell^{2}(\mathbb{Z}).
\end{equation}

We will assume in the following that the reaction function $ g $ in~\eqref{e:lde:Nagumo} is defined on whole $ \mathbb{R} $, satisfies~\ref{hyp:g:nodes}--\ref{hyp:g:shape},
\begin{enumerate}[label=\textit{(g\textsubscript{\arabic*})}]
\setcounter{enumi}{2}
\item \label{hyp:g:bounded-extended-signs} $ g(u) > 0 $ for $ u < 0 $ and $ g(u) < 0 $ for $ u > 1 $,
\end{enumerate}
and $ g \in C^{2}(\mathbb{R}) $ with a bounded second derivative:
\begin{enumerate}[label=\textit{(g\textsubscript{\arabic*})}]
\setcounter{enumi}{3}
\item \label{hyp:g:bounded-2-derivative} there exists $ M > 0 $ such that $ |g''(u)| \leq  M $ for all $ u \in \mathbb{R} $.
\end{enumerate}
We can state these assumptions without loss of generality because we are interested in solutions of~\eqref{e:lde:Nagumo} with values contained in the interval $ [0,1] $ which is invariant thanks to the following maximum principle (i.e., one can suitably extend the reaction function defined on $ [0,1] $).

\begin{thm}[Existence and uniqueness of $ \ell^{\infty} $-solution, interval invariance]\label{t:stab:invariance}
Let $ g \in C^{1}(\mathbb{R}) $ satisfy~\ref{hyp:g:nodes}--\ref{hyp:g:bounded-extended-signs}. Then for every $ \chi^{0} = (\chi_{i}^{0}) \in \ell^{\infty}(\mathbb{Z}) $ there exists a unique classical solution $ \chi: [0,\infty) \to \ell^{\infty}(\mathbb{Z}) $ of
\begin{equation}\label{e:stab:cauchy_chi-elementwise}
\begin{cases}
\chi'(t) = d(\chi_{i-1}(t)-2\chi_{i}(t)+\chi_{i+1}(t)) + g(\chi_{i}(t)), \quad t \geq 0,\\
\chi(0)=\chi^0,
\end{cases}
\end{equation}
and if $ \chi_{i}^{0} \in [\alpha,\beta] $ for all $ i \in \mathbb{Z} $ with $ \alpha \leq 0 $ and $ \beta \geq 1 $ then
\begin{equation}\label{e:stab:WMP}
\chi_{i}(t) \in [\alpha,\beta] \quad \text{for all} \quad t \geq 0 \quad \text{and all} \quad i \in \mathbb{Z}.
\end{equation}
\end{thm}

\begin{proof}
The statement follows from Slav\'{i}k~et~al.~\cite{Slavik2019} -- Theorems~4.4 (weak maximum principle) and 4.6 (global existence).
\end{proof}




Let us define the linear operator $ L : \ell^{\infty}(\mathbb{Z}) \to \ell^{\infty}(\mathbb{Z}) $ by
\begin{equation}\label{e:stab:linear-oper-diffusion}
(Lu)_{i} = d( u_{i-1} -2u_{i} + u_{i+1} ), \quad i \in \mathbb{Z},
\end{equation}
(note that $ Lu \in \ell^{2}(\mathbb{Z}) $ provided $ u \in \ell^{2}(\mathbb{Z})$ as well) and the nonlinear operator $ G : \ell^{\infty}(\mathbb{Z}) \to \ell^{\infty}(\mathbb{Z}) $ by
\begin{equation}\label{e:stab:nonlinear-oper-reaction}
(G(u))_{i} = g(u_{i}), \quad i \in \mathbb{Z},
\end{equation}
in which $ G(u) \in \ell^{\infty}(\mathbb{Z}) $ follows from the continuity of $ g $. Then we rewrite~\eqref{e:stab:cauchy_chi-elementwise} into the operator form in $ \ell^{\infty}(\mathbb{Z}) $:
\begin{equation}\label{e:stab:cauchy_chi}
\begin{cases}
{\chi}'(t) = L \chi(t) + G\left(\chi(t)\right), \quad t \geq 0,\\
\chi(0)=\chi^0=u+\psi^0.
\end{cases}
\end{equation}

Let us define $ \psi(t) = \chi(t) - u $ and linearize~\eqref{e:stab:cauchy_chi} at the stationary solution $ u $.

\begin{lem}[Linearization of $ G $]\label{l:stab:G-smoothness}
Let $ g \in C^{1}(\mathbb{R}) $. The operator $ G : \ell^{\infty}(\mathbb{Z}) \to \ell^{\infty}(\mathbb{Z}) $ defined by~\eqref{e:stab:nonlinear-oper-reaction} is continuously differentiable at every $ u \in \ell^{\infty}(\mathbb{Z}) $ so that
\begin{equation}\label{e:stab:linearization-G}
G(u+\psi) = G(u) + G'(u)\psi + \omega(u,\psi) \quad \text{for all} \quad \psi \in \ell^{\infty}(\mathbb{Z}),
\end{equation}
in which the differential is given as
\begin{equation}\label{e:stab:differential-G}
(G'(u)\psi)_{i} = g'(u_{i})\psi_{i}, \quad i \in \mathbb{Z},
\end{equation}
and the nonlinear remainder $ \omega(u,\cdot) $ is $ o(\|\psi\|_{\ell^{\infty}}) $.
\end{lem}

\begin{proof}
We can proceed by a direct computation of the differential~\eqref{e:stab:differential-G} for $ u, \psi \in \ell^{\infty}(\mathbb{Z}) $ and $ s > 0 $
\[
(G'(u)\psi)_{i} = \lim\limits_{s\to 0} \frac{g(u_{i}+s\psi_{i}) - g(u_{i})}{s} = \lim\limits_{s\to 0} g'(u_{i}+\vartheta_{i}(s))\psi_{i} = g'(u_{i})\psi_{i},
\]
by the mean value theorem ($ 0 \leq |\vartheta_{i}(s)| \leq |s\psi_{i}| \to 0 $ as $ s \to 0 $). The convergence in $ \ell^{\infty}(\mathbb{Z}) $ and continuity of $ G'(\cdot) $ in $ \ell^{\infty}(\mathbb{Z}) $ then follow from the uniform continuity of $ g' $ on compact intervals.
\end{proof}

We show that the linear part of the problem generates a $ C_{0} $-semigroup in $ \ell^{2}(\mathbb{Z}) $ and that the nonlinear remainder $ \omega(u,\cdot) $ is well-defined on $ \ell^{2}(\mathbb{Z}) $, is continuously differentiable in $ \ell^{2} $-norm, and is subquadratic.

\begin{lem}[Linear part $ L + G'(u) $]\label{l:stab:linear-part}
Let $ u \in \ell^{\infty}(\mathbb{Z}) $, $ g \in C^{1}(\mathbb{R}) $, $ L: \ell^{2}(\mathbb{Z}) \to \ell^{2}(\mathbb{Z}) $ be defined by~\eqref{e:stab:linear-oper-diffusion}, and $ G'(u): \ell^{2}(\mathbb{Z}) \to \ell^{2}(\mathbb{Z}) $ be defined by~\eqref{e:stab:differential-G}
. Then the linear operator $ A = L+G'(u) $ is bounded in $ \ell^{2}(\mathbb{Z}) $ and thus, a generator of a $ C_{0} $-semigroup in $ \ell^{2}(\mathbb{Z}) $.
\end{lem}

\begin{proof}
Let $ m = \sup_{i \in \mathbb{Z}} |g'(u_i)|$ and $ \psi \in \ell^{2}(\mathbb{Z}) $. Then, Minkowski's inequality yields
\begin{align*}
\| A\psi\|^2_{\ell^2} & =  \sum_{i \in \mathbb{Z}}|d\left(\psi_{i-1} - 2\psi_{i} + \psi_{i+1}\right) + g'(u_i)\psi_i|^2\\  &\leq  \sum_{i \in \mathbb{Z}}|d\left(\psi_{i-1} - 2\psi_{i} + \psi_{i+1}\right)|^2 + \sum_{i \in \mathbb{Z}}| g'(u_i)\psi_i|^2 \\ &\leq  \left( 16d^{2}  + m^{2} \right) \|\psi\|^2_{\ell^2},
\end{align*}
i.e., the operator $ A $ is bounded and thus generates a uniformly continuous semigroup
\begin{equation}\label{e:semigroup_T}
T(t)=e^{tA}=\sum_{n=0}^{\infty} \frac{(tA)^n}{n!}.
\end{equation}
A uniformly continuous semigroup $T(t)$, $t>0$, of a bounded linear operator on $\ell^{2}(\mathbb{Z}) $ forms a $C_0$-semigroup, i.e., it satisfies $ \lim_{t \to 0^-}T(t)\psi = \psi $ for every $ \psi \in \ell^2(\mathbb{Z}) $.
\end{proof}

Let us note that Lemmas~\ref{l:stab:G-smoothness} and~\ref{l:stab:linear-part} are valid even if the function $ g $ is of class $ C^{1} $. However, the following statement controlling the nonlinear remainder needs the $ C^{2} $-function with a bounded second derivative.

\begin{lem}[Higher order term $ \omega(u,\cdot) $]\label{l:stab:omega-l2}
Let $ u \in \ell^{\infty}(\mathbb{Z}) $, $ g \in C^{2}(\mathbb{R}) $ satisfies~\ref{hyp:g:bounded-2-derivative}, and $ \omega(u,\cdot) $ be the nonlinear remainder from Lemma~\ref{l:stab:G-smoothness}. Then $ \omega(u,\psi) \in \ell^{2}(\mathbb{Z}) $ provided $ \psi \in \ell^{2}(\mathbb{Z}) $, $ \omega(u,\cdot) $ is continuously differentiable in $ \ell^{2} $-norm, and there exists $ C > 0 $ such that
\begin{equation}\label{e:stab:nonlinear-bound-l2}
\| \omega(u,\psi) \|_{\ell^{2}} \leq C \| \psi \|_{\ell^{2}}^{2} \quad \text{for all} \quad \psi \in \ell^{2}(\mathbb{Z}).
\end{equation}
\end{lem}

\begin{proof}
If $ \psi \in \ell^{2}(\mathbb{Z}) $, then $ \|\psi\|_{\ell^{\infty}} \leq c_{1} \|\psi\|_{\ell^{4}} \leq c_{1} c_{2} \| \psi \|_{\ell^{2}} $ ($ c_{1}, c_{2} > 0 $ independent on $ \psi $), since $ \ell^{2}(\mathbb{Z}) \hookrightarrow \ell^{4}(\mathbb{Z}) \hookrightarrow \ell^{\infty}(\mathbb{Z}) $. Then using~\eqref{e:stab:linearization-G}, the mean value theorem twice, and~\ref{hyp:g:bounded-2-derivative} we obtain for $ \psi \in \ell^{2}(\mathbb{Z}) $ 
\begin{align*}
\| \omega(u,\psi) \|_{\ell^{2}}^{2} & = \sum\limits_{i\in\mathbb{Z}}  | g(u_{i}+\psi_{i}) - g(u_{i}) - g'(u_{i}) \psi_{i} |^{2} = \sum\limits_{i\in\mathbb{Z}}  | g''(u_{i}+\vartheta_{i}) |^{2} |\psi_{i}|^{4}
\leq M^{2} \| \psi \|_{\ell^{4}}^{4} \leq M^{2} c_{2}^{4} \| \psi \|_{\ell^{2}}^{4},
\end{align*}
in which $ -|\psi_{i}| \leq \vartheta_{i} \leq |\psi_{i}| $. This proves that $ \omega(u,\psi) \in \ell^{2}(\mathbb{Z}) $ as well as the inequality~\eqref{e:stab:nonlinear-bound-l2}. We compute, the differential of $ \omega(u,\cdot) $ for $ \psi, \xi \in \ell^{2}(\mathbb{Z}) $ as
\begin{align*}
(\omega'(u,\psi)\xi)_{i} & = \lim\limits_{s\to 0} \frac{g(u_{i}+\psi_{i}+s\xi_{i})-g'(u_{i})(\psi_{i}+s\xi_{i}) - g(u_{i}+\psi_{i})+g'(u_{i})\psi_{i}}{s} \\
& = \lim\limits_{s\to 0} \left( g'(u_{i}+\psi_{i}+\vartheta_{i}(s))\xi_{i} - g'(u_{i})\xi_{i} \right) \\
& = (g'(u_{i}+\psi_{i}) - g'(u_{i}))\xi_{i},
\end{align*}
in which $ 0 \leq |\vartheta_{i}(s)| \leq |s\xi_{i}| \to 0 $. The convergence in the $ \ell^{2} $-norm and the continuity of $ \omega'(u,\cdot) $ follows again by the mean value theorem, the H\"{o}lder's inequality, the embeddings $ \ell^{2}(\mathbb{Z}) \hookrightarrow \ell^{4}(\mathbb{Z}) \hookrightarrow \ell^{\infty}(\mathbb{Z}) $, and the boundedness of $ g'' $ by~\ref{hyp:g:bounded-2-derivative}.
\end{proof}

Consequently, we can finally show that $ \chi(t) $ is the classical solution of~\eqref{e:stab:cauchy_chi} on $ \ell^{\infty}(\mathbb{Z}) $ whenever $ \psi(t) = \chi(t) - u $ is the classical solution of the following problem in $ \ell^{2}(\mathbb{Z}) $:
\begin{equation}\label{e:stab:cauchy_psi}
\begin{cases}
{\psi}'(t) = L \psi(t) + G'(u)\psi(t) + \omega(u,\psi(t)), \quad t \geq 0,\\
\psi(0)=\psi^0.
\end{cases}
\end{equation}
The following claim formalizes the existence and uniqueness of the classical solution of the $ \ell^{2} $-problem~\eqref{e:stab:cauchy_psi}.

\begin{thm}[Existence and uniqueness for $ \ell^{2} $-perturbation]\label{t:stab:EandU-l2}
Let $ u \in \ell^{\infty}(\mathbb{Z}) $ and $ g \in C^{2}(\mathbb{R}) $ satisfies~\ref{hyp:g:bounded-2-derivative}. Then there exists a unique classical solution $ \psi: [0,t_{\max}) \to \ell^{2}(\mathbb{Z}) $ of~\eqref{e:stab:cauchy_psi} and either $ t_{\max} = \infty $, or $ \|\psi(t)\|_{\ell^{2}} \to \infty $ as $ t \to t_{\max} $.
\end{thm}

\begin{proof}
The statement is an immediate consequence of Lemmas~\ref{l:stab:linear-part}, \ref{l:stab:omega-l2} and Pazy~\cite[Theorem~6.1.5]{Pazy1983}.
\end{proof}

\begin{lem}[Solutions equivalence]\label{l:stab:equivalence}
Let $ u \in \ell^{\infty}(\mathbb{Z}) $ be a stationary solution of~\eqref{e:stab:cauchy_chi} (i.e., $ Lu + G(u) = 0 $), $ \psi^{0} \in \ell^{2}(\mathbb{Z}) $, and $ g \in C^{2}(\mathbb{R}) $ satisfies~\ref{hyp:g:nodes}--\ref{hyp:g:bounded-2-derivative}. Then, $ \psi(t) \in \ell^{2}(\mathbb{Z}) $, $ t \in [0,t_{\max}) $, is the unique classical solution of~\eqref{e:stab:cauchy_psi} on $ \ell^{2}(\mathbb{Z}) $ if and only if $ \chi(t) = u + \psi(t) \in \ell^{\infty}(\mathbb{Z}) $, $ t \in [0,t_{\max}) $, is the unique classical solution of~\eqref{e:stab:cauchy_chi} on $ \ell^{\infty}(\mathbb{Z}) $.
\end{lem}

\begin{proof}
Let $ \psi(t) \in \ell^{2}(\mathbb{Z}) $, $ t \in [0,t_{\max}) $, be the unique classical solution of~\eqref{e:stab:cauchy_psi} on $ \ell^{2}(\mathbb{Z}) $ and $ \chi(t) = u+\psi(t) $, $ t \in [0,t_{\max}) $. Then immediately $ \chi(t) \in \ell^{\infty}(\mathbb{Z}) $ because $ \ell^{2}(\mathbb{Z}) \subset \ell^{\infty}(\mathbb{Z}) $. Moreover, applying~$ Lu + G(u) = 0 $ and~Lemma~\ref{l:stab:G-smoothness} we get
\begin{align*}
\chi'(t) = (u+\psi(t))' = \psi'(t) & = L\psi(t) + G'(u)\psi(t) + \omega(u,\psi(t)) \\
& = Lu + G(u) + L\psi(t) + G'(u)\psi(t) + \omega(u,\psi(t)) \\
& = L(u+\psi(t)) + G(u+\psi(t)) \\
& = L\chi(t) + G(\chi(t))
\end{align*}
and $ \chi(0) = u + \psi^{0} $. Thus, $ \chi(t) = u + \psi(t) $ is the unique classical solution of~\eqref{e:stab:cauchy_chi} on $ \ell^{\infty}(\mathbb{Z}) $.

Let $ \chi(t) \in \ell^{\infty}(\mathbb{Z}) $, $ t \geq 0 $, be the unique classical solution of~\eqref{e:stab:cauchy_chi} on $ \ell^{\infty}(\mathbb{Z}) $ and $ \psi(t) = \chi(t) - u $, $ t \geq 0 $. Then, we obtain by Lemma~\ref{l:stab:G-smoothness} and~$ Lu + G(u) = 0 $ again
\begin{align*}
\psi'(t) = (\chi(t)-u)' = \chi'(t) = L\chi(t) - G(\chi(t)) & = L(u + \psi(t)) + G(u+\psi(t)) \\
& = Lu + L\psi(t) + G(u)+ G'(u)\psi(t) + \omega(u,\psi(t))\\
& = L\psi(t) + G'(u)\psi(t) + \omega(u,\psi(t)),
\end{align*}
and $ \psi(0) = \psi^{0} $, i.e., $ \psi(t) = \chi(t) - u $ satisfies~\eqref{e:stab:cauchy_psi}. We finally verify that $ \psi(t) \in \ell^{2}(\mathbb{Z}) $, $ t \in [0,t_{\max}) $. Indeed, the problem~\eqref{e:stab:cauchy_psi} has a unique classical solution $ \tilde{\psi}(t) \in \ell^{2}(\mathbb{Z}) $, $ t \in [0,t_{\max}) $, by Theorem~\ref{t:stab:EandU-l2}. Since $ \ell^{2}(\mathbb{Z}) \subset \ell^{\infty}(\mathbb{Z}) $, then analogically as in the first part of the proof, $ \tilde{\chi}(t) = u + \tilde{\psi}(t) $, $ t \in [0,t_{\max}) $, is a classical solution of~\eqref{e:stab:cauchy_chi}. Since~\eqref{e:stab:cauchy_chi} has the unique classical solution on $ [0,\infty) $, then there has to necessarily be $ \chi(t) = \tilde{\chi}(t) \in \ell^{\infty}(\mathbb{Z}) $, i.e., $ \psi(t) = \tilde{\psi}(t) \in \ell^{2}(\mathbb{Z}) $, $ t \in [0,t_{\max}) $.
\end{proof}

Lemma~\ref{l:stab:equivalence} justifies the following definition of $ \ell^{2} $-stability for~\eqref{e:lde:Nagumo} (or the abstract problem~\eqref{e:stab:cauchy_chi}).

\begin{defn}[$ \ell^{2} $-stability]\label{d:stab:stability}
Let $ u \in \ell^{\infty}(\mathbb{Z}) $ be a stationary solution of~\eqref{e:lde:Nagumo} (i.e., of~\eqref{e:stab:cauchy_chi} as well) and $ g \in C^{2}(\mathbb{R}) $ satisfies~\ref{hyp:g:nodes}--\ref{hyp:g:bounded-2-derivative}.
\begin{enumerate}
\item The stationary solution $ u \in \ell^{\infty}(\mathbb{Z}) $ is \emph{locally $ \ell^{2} $-stable} if for every $\varepsilon>0$ there exists $\delta>0$ such that for every $ \psi^{0} \in \ell^{2}(\mathbb{Z}) $ satisfying $\| \psi^{0} \|_{\ell^{2}} <\delta $ the unique solution $ \psi(t) \in \ell^{2}(\mathbb{Z}) $, $ t \geq 0 $, of~\eqref{e:stab:cauchy_psi} satisfies
\[
\| \psi(t) \|_{\ell^{2}} < \varepsilon \quad \text{for all} \quad t \in [0,\infty).
\]
\item The stationary solution $ u \in \ell^{\infty}(\mathbb{Z}) $ is \emph{locally asymptotically $ \ell^{2} $-stable} if it is locally $ \ell^{2} $-stable and \emph{attractive}, i.e., if there exists $\gamma > 0$ such that for every $ \psi^{0} \in \ell^{2}(\mathbb{Z}) $ satisfying $\| \psi^{0} \|_{\ell^{2}} < \gamma $ the unique solution $ \psi(t) \in \ell^{2}(\mathbb{Z}) $, $ t \geq 0 $, of~\eqref{e:stab:cauchy_psi} satisfies
\[
\| \psi(t) \|_{\ell^{2}} \to 0 \quad \text{as} \quad t \to \infty.
\]
\item The stationary solution $ u \in \ell^{\infty}(\mathbb{Z}) $ is \emph{$ \ell^{2} $-unstable} if it is not locally $ \ell^{2} $-stable.
\end{enumerate}
\end{defn}

We have finally collected all necessary concepts to prove Theorem~\ref{t:main-stability}.

\begin{proof}[Proof of Theorem~\ref{t:main-stability}]
Let $ u = (u_i) \in \ell^{\infty}(\mathbb{Z}) $ be a stationary solution of the LDE~\eqref{e:lde:Nagumo} corresponding to $s\in\Sigma_3$. Since $ g \in C^{2}([0,1]) $, we are interested only in the solutions of~\eqref{e:lde:Nagumo} that possess values in the interval $ [0,1] $ which is invariant by Theorem~\ref{t:stab:invariance} ($ \alpha = 0 $, $ \beta = 1 $). Therefore, let us extend, without loss of generality, $ g $ outside the interval $ [0,1] $ to be defined on $ \mathbb{R} $ and satisfy~\ref{hyp:g:bounded-extended-signs}--\ref{hyp:g:bounded-2-derivative}.

Let $ L: \ell^{2}(\mathbb{Z}) \to \ell^{2}(\mathbb{Z}) $ be defined by~\eqref{e:stab:linear-oper-diffusion} and $ G'(u): \ell^{2}(\mathbb{Z}) \to \ell^{2}(\mathbb{Z}) $ be the differential~\eqref{e:stab:differential-G} of $ G $ at the stationary solution $ u \in \ell^{\infty}(\mathbb{Z}) $ of~\eqref{e:lde:Nagumo}. By Lemma~\ref{l:stab:linear-part}, the operator $ A = L + G'(u) $ is linear and bounded in the Hilbert space $ \ell^{2}(\mathbb{Z}) $. Moreover, $ A $ is self-adjoint, i.e., $\langle A \psi, \psi\rangle_{\ell^2} = \langle \psi, A \psi \rangle_{\ell^2}$ for all $ \psi \in \ell^{2}(\mathbb{Z}) $. Thus, the spectrum\footnote{In this section, we use the standard notation for the spectrum, denoted by $\sigma$. This should not be confused with the shift on biinfinite sequences $\Sigma_3$ introduced in Sec.~\ref{sec:symbolic:Moser} and which does not appear in the present section. Likewise, we denote the resolvent set by $\rho$, which should not be confused with the reflection map introduced in the proof of Lemma~\ref{l:vertical:strips:onto:horizontal}. } $\sigma(A) \subset \mathbb{R}$ and by \cite[Theorem~2.20]{Teschl2014}
\begin{equation}\label{e:stab:estimate_teschl}
\sup \sigma(A) =\sup_{\|\psi\|_{\ell^2}=1}\langle A \psi, \psi\rangle_{\ell^2}.  
\end{equation}
Finally, the nonlinear remainder $\omega(u,\cdot): \ell^{2}(\mathbb{Z}) \to \ell^{2}(\mathbb{Z}) $ is continuously differentiable on $ \ell^{2}(\mathbb{Z})$ and grows subquadratically by~Lemma~\ref{l:stab:omega-l2}.

Consequently, the $ \ell^{2} $-stability of $ u $ is determined by the position of spectrum $ \sigma(A) $ by~\cite[Theorems~5.1.1 and~5.1.3]{Henry1981}. Specifically, if there exists $ m > 0 $ such that $ \sigma(A) \subset (-\infty, -m] $ then $ u $ is locally asymptotically $ \ell^{2} $-stable; if $\sigma(A) \cap (0,\infty) \neq \emptyset $ then $ u $ is $ \ell^{2} $-unstable.

    \begin{enumerate}[label=(\roman*)]
        \item If $u_i \in [0,a_1) \cup (a_2,1]$ (or equivalently $s_i\in\{0,1\}$) for all $ i \in \mathbb{Z} $, then $u_i \in [0,\ru_2] \cup [\ru_6,1]$ for all $i \in \mathbb{Z}$ by the construction of sectors $ \sector_w $. Thus, $g'(u_i)\leq -m <0$ for all $i \in \mathbb{Z}$ and we estimate for $ \psi \in \ell^2$
        \begin{align}
        \begin{split}\label{e:stab:estimate_unit_product}
            \langle A\psi,\psi \rangle_{\ell^2} & = \langle L\psi,\psi \rangle_{\ell^2} + \langle G'(u)\psi,\psi \rangle_{\ell^2}\\[1ex]
            & = d \sum_{i \in \mathbb{Z}}\left(\psi_{i-1}\psi_i -2\psi_i^2 + \psi_i\psi_{i+1}\right) + \sum_{i \in \mathbb{Z}} g'(u_i)\psi_i^2\\[1ex]
            & \leq -d \sum_{i \in \mathbb{Z}} \left(\psi_i-  \psi_{i+1}\right)^2  -m \sum_{i \in \mathbb{Z}} \psi_i^2\\[1ex]
            & \leq -m \|\psi\|^2_{\ell^2},
        \end{split}
        \end{align}
        which yields $ \sigma(A) \subset (-\infty, -m] $ and $ u $ is locally asymptotically $ \ell^{2} $-stable (it is locally asymptotically stable with respect to all initial conditions, i.e., also for those with values in $ [0,1] $ where the extension of $ g $ does not play any role).
        \item If there exists an index $i_0 \in \mathbb{Z}$ such that $u_{i_0} \in (a_1,a_2)$ (or equivalently $s_{i_0}=a$), then $u_{i_0} \in [\ru_3,\ru_5]$ (by the construction of sectors $ \sector_w $) and $g'(u_{i_0})>0$. Let $\psi = e_{i_0}$, i.e., $\psi_{i_0} =1$ and $\psi_i=0$ if $i \not= i_0$. Then $\|\psi\|_{\ell^2}=1$ and
        \[
        \langle A \psi, \psi\rangle_{\ell^2} = \langle A e_{i_0}, e_{i_0} \rangle_{\ell^2}=-2d + g'(u_{i_0})>0,
        \]
        i.e., $ \sigma(A) \cap (0,\infty) \neq \emptyset $ and $ u $ is $ \ell^{2} $-unstable (recall that $ u_{i} \in (0,1) $ for all $ i \in \mathbb{Z} $ in this case, i.e., $ u $ is unstable with respect to the initial conditions with values in $ [0,1] $, i.e., the extension of $ g $ does not play any role again). \qedhere
    \end{enumerate}
\end{proof}

Alternatively, we can consider the stability of stationary states $u = (u_i) \in \ell^{\infty}(\mathbb{Z}) $ of \eqref{e:lde:Nagumo} with respect to bounded perturbations $\psi^0 \in \ell^{\infty}(\mathbb{Z}) $ instead of $\psi^0 \in \ell^{2}(\mathbb{Z}) $. Naturally, a slightly different approach is required, since we work in a Banach space in this case.


\begin{rmk}[$\ell^{\infty}$-stability]
    Both Lemmas~\ref{l:stab:linear-part} and~\ref{l:stab:omega-l2} remain valid when the Hilbert space $ \ell^{2}(\mathbb{Z}) $ is replaced by the Banach space $ \ell^{\infty}(\mathbb{Z}) $. In this case, we can obtain the estimates
    \[
    \|A\psi\|_{\ell^\infty} \leq (4d + m) \|\psi\|_{\ell^\infty} \quad \text{ and } \quad \|\omega(u,\psi)\|_{\ell^\infty} \leq M \|\psi\|_{\ell^\infty}^2
    \]
    in the norm of $ \ell^{\infty}(\mathbb{Z}) $. In general, the linear operator $ A = L+ G'(u) $  is no longer self-adjoint and the spectrum $\sigma(A) \subset \mathbb{C}$. However, we show that $\sigma_{\ell^2}(A)=\sigma_{\ell^\infty}(A)$.
    
    The linear operator $ A $ is a band-dominated operator~\cite[Definition~2.1.5]{Rabinovitch2004}. In particular, defining $ \alpha^{-1}, \alpha^{1} \in \ell^{\infty}(\mathbb{Z})$ such that $ \alpha^{-1} = \alpha^{1} \equiv d $ and $ \alpha^{0} \in \ell^{\infty}(\mathbb{Z})$ such that $ \alpha^0_i = g'(u_i)-2d $ (recall that $g \in C^{1}(\mathbb{R})$) we have
    \[
    A = \sum_{j \in \{-1,0,1\}} \alpha^j S^j,
    \]
    where $S^j$, $ j \in \{-1,0,1\} $ are shifts $(S^j u)_i = u_{i-j}$, $ i \in \mathbb{Z} $. Thus, both $ A $ and $ A_\lambda := A-\lambda I $ for all $ \lambda \in \mathbb{C} $ are elements of the Wiener algebra~\cite[Section~2.5]{Rabinovitch2004}.
    

    Let $\lambda \in \rho_{\ell^2}(A)$, where $ \rho_{\ell^2}(A)$ denotes the resolvent of $ A $ in $ \ell^2(\Z) $. Thus $A_\lambda^{-1}$ exists and is bounded, i.e., $\| A_\lambda^{-1} \|_{\ell^2}<\infty$. By~\cite[Corollary~2.5.4]{Rabinovitch2004} if $A_\lambda$ is invertible in a specific $\ell^p(\mathbb{Z})$, $1\leq p \leq \infty$, then it is invertible in every $\ell^{p}(\mathbb{Z}) $, $1 \leq p \leq \infty$, with uniformly bounded inverses. It follows that $\lambda \in \rho_{\ell^p}(A)$ for every $1 \leq p \leq \infty$. Therefore, both the resolvent $\rho(A)$ and the spectrum $\sigma(A)$ are the same for $\ell^2(\mathbb{Z})$ and $\ell^\infty(\mathbb{Z})$.
    
    Finally, $\ell^\infty$-stability and $\ell^\infty$-instability of a steady state $ u \in \ell^{\infty}(\mathbb{Z}) $ of \eqref{e:stab:cauchy_chi} again follows by~\cite[Theorems~5.1.1 and~5.1.3]{Henry1981}.
\end{rmk}



\section{Conclusion}\label{sec:discussion}
We conclude our paper by discussing two natural extensions of our results. We first discuss extension to Frenkel-Kontorova-type models and then weaken assumptions on the nonlinearity $g$.

\paragraph{Frenkel-Kontorova models} Since Theorems~\ref{t:main-uniqueness} and~\ref{t:main-twosol} deal with stationary solutions, they can straightforwardly be extended to the Frenkel-Kontorova-type lattice equations
\begin{equation}\label{e:lde:FK} 
{u}_{i}''(t) = d(u_{i-1}(t) - 2 u_{i}(t) + u_{i+1}(t))+g(u_{i}(t);a),\quad i\in\mathbb{Z}, \quad t \geq 0,
\end{equation}
and to the Frenkel-Kontorova-type lattice equations with damping $\gamma>0$
\begin{equation}\label{e:lde:FK-damping} 
{u}_{i}''(t) + \gamma u_{i}'(t)= d(u_{i-1}(t) - 2 u_{i}(t) + u_{i+1}(t))+g(u_{i}(t);a),\quad i\in\mathbb{Z}, \quad t \geq 0.
\end{equation}
Note that stationary solutions of \eqref{e:lde:FK}--\eqref{e:lde:FK-damping} satisfy the second-order difference equation~\eqref{e:stationary} and could thus be studied by the planar map $\phi$~\eqref{e:map:Keener} and Theorems~\ref{t:main-uniqueness} and~\ref{t:main-twosol} follow. It is fair to emphasize that this extension has two caveats. First, the stability result Theorem~\ref{t:main-stability} is no longer extendable, since \eqref{e:lde:FK} is conservative. Next, our analysis describes only a small part of equilibria for the standard Frenkel-Kontorova model with the common multistable reaction function $g(u)=\sin(u)$, \cite{Floria2005}.

\paragraph{Weakened assumptions on $g$} We assume that $g$ is a smooth function satisfying \ref{hyp:g:nodes}--\ref{hyp:g:shape} in Theorems~\ref{t:main-uniqueness}--\ref{t:main-stability}. Our analysis of stationary solutions is concentrated to the square $Q=[0,1]^2$ and the perturbations studied in Section~\ref{sec:stability} are also confined to the interval $[0,1]$ by the interval invariance in Theorem~\ref{t:stab:invariance}. Since our statements focus on sufficiently small $d$ we may consider bistable reaction functions $g$ that are smooth (or $C^2$ in the case of Theorem~\ref{t:main-stability}) only in the neighborhood of its roots~$0,a,1$ and replace \ref{hyp:g:shape} by
\begin{enumerate}        
    \item[$\emph{(}{g}'_2\emph{)}$] \label{hyp:g:shape:2} $ g \in C^{1}([0,\varepsilon]\cup [a-\varepsilon,a+\varepsilon] \cup[1-\varepsilon,1]) $ for an $ \varepsilon > 0 $, $ g'(0) < 0 $, $ g'(a) > 0 $, and $ g'(1) < 0 $. 
\end{enumerate}
Consequently, these generalizations of Theorems~\ref{t:main-uniqueness}--\ref{t:main-stability} include also Elmer's results \cite{Elmer2006, Elmer2005} for the sawtooth bistability $g=g_\mathrm{ST}$ from~\eqref{e:sawtooth}. In other words, we may allow nonsmoothness or nonmonotonicity if they occur outside the unit square $Q=[0,1]^2$ for the auxiliary function $h(u;a,d)$ \eqref{e:h}, see Figures~\ref{fig:strips:Vi} and~\ref{fig:1sectors}.

\section*{Acknowledgments}
\label{sec:acknowledgments}
\noindent JH acknowledges the support of the grant SGS-2025-007 by the Czech Ministry of Education, Youth and Sports.
The authors are grateful to Hermen Jan Hupkes, Anton\'{\i}n Slav\'{\i}k, and Vladim\'{\i}r \v{S}v\'{\i}gler for helpful comments and discussions.

\nocite{Elmer2007}
{
\footnotesize
\bibliographystyle{abbrv}
\bibliography{twosol_lit}
}

\end{document}